\documentclass[11pt]{article}
\usepackage{amsmath,amssymb,amsthm}
\usepackage[ruled,boxed]{algorithm}
\usepackage{algorithmic}
\usepackage[margin=1.6in]{geometry}

\newcommand{\tensor}[1]{\mathcal{#1}}

\newcommand{\abs}[1]{\left|#1\right|}
\newcommand{\norm}[1]{\left\|#1\right\|}

\newcommand{\R}{\mathbb{R}}

\newcommand{\N}{\mathbb{N}}

\newcommand{\E}{\mathbb{E}}
\newcommand{\Prob}{\mathbb{P}}
\newcommand{\V}{\mbox{\bf Var}}
\newcommand{\Sphere}{\mathbb{S}}
\newcommand{\sgn}[1]{\textbf{sign}\left({#1}\right)}

\newtheorem{thm}{Theorem}
\newtheorem{corol}{Corollary}

\newtheorem{lem}{Lemma}
\newtheorem{definition}{Definition}

\usepackage{hyperref}
\hypersetup{colorlinks=true}

\allowdisplaybreaks

\begin{document}
\date{}

\title{Tensor sparsification via a bound on the spectral norm of random tensors}

\author{
Nam Nguyen
\thanks{
Department of Mathematics,
Massachusetts Institute of Technology,
namnguyen@math.mit.edu. 
Most of the work was done while the author was a graduate student at Johns Hopkins University.
}
\and
Petros Drineas
\thanks{
Department of Computer Science,
Rensselaer Polytechnic Institute,
drinep@cs.rpi.edu.
}
\and
Trac Tran
\thanks{
Department of Electrical and Computer Engineering,
The Johns Hopkins University,
trac@jhu.edu.
}
}

\maketitle

\begin{abstract}
{Given an order-$d$ tensor $\tensor A \in \R^{n \times n \times \ldots \times n}$, we present a simple, element-wise sparsification algorithm that zeroes out all sufficiently small elements of $\tensor A$, keeps all sufficiently large elements of $\tensor A$, and retains some of the remaining elements with probabilities proportional to the square of their magnitudes. We analyze the approximation accuracy of the proposed algorithm using a powerful inequality that we derive. This inequality bounds the spectral norm of a random tensor and is of independent interest. As a result, we obtain novel bounds for the tensor sparsification problem. }%As an added bonus, we obtain improved bounds for the matrix ($d=2$) sparsification problem.}
{tensor; tensor norm;tensor sparsification;fast tensor computation;random tensor}
\end{abstract}

\section{Introduction}

Technological developments over the last two decades (in both scientific and internet domains) permit the automatic generation of very large data sets. Such data are often modeled as matrices, since an $m \times n$ real-valued matrix $A$ provides a natural structure to encode information about $m$ objects, each of which is described by $n$ features. A generalization of this framework permits the modeling of the data by higher-order arrays or tensors (e.g., arrays with more than two modes). A natural example is time-evolving data, where the third mode of the tensor represents time~\cite{FKS07}. Numerous other examples exist, including tensor applications in higher-order statistics, where tensor-based methods have been leveraged in the context of, for example, Independent Components Analysis (ICA), in order to exploit the statistical independence of the sources~\cite{DDV00a,DDV00b,DDV00c}.

A large body of recent work has focused on the design and analysis of algorithms that efficiently create small ``sketches'' of matrices and tensors. By sketches, we mean a new matrix or tensor with significantly smaller size than the original ones. Such sketches are subsequently used in eigenvalue and eigenvector computations~\cite{FKV98,AM01}, in data mining applications~\cite{MMD06,MMD08,DM07,MD09}, or even to solve combinatorial optimization problems~\cite{ADKK03,DKKV05,DKM08}. Existing approaches include, for example, the selection of a small number of rows and columns of a matrix in order to form the so-called CUR matrix/tensor decomposition~\cite{DMM08a,MMD06,MMD08}, as well as random-projection-based methods that employ fast randomized variants of the Hadamard-Walsh transform~\cite{Sar06} or the Discrete Cosine Transform~\cite{NDT09}.

An alternative approach was pioneered by Achlioptas and McSherry in 2001~\cite{AM01,AM07} and leveraged the selection of a small number of elements in order to form a sketch of the input matrix. A rather straight-forward extension of their work to tensors was described by Tsourakakis in~\cite{Tso09}. Another remarkable direction was pioneered in the work of Spielman, Teng, Srivastava, and collaborators~\cite{SS08,BSS09}, who proposed algorithms for graph sparsification in order to create preconditioners for systems of linear equations with Laplacian input matrices. Partly motivated by their work, we define the following matrix/tensor sparsification problem:
\begin{definition}\label{def:matrixsparsification}\textsc{[Matrix/tensor Sparsification]}
Given an order-$d$ tensor $\tensor A \in \R^{n \times n \times \ldots \times n}$ and an error parameter $\epsilon \geq 0$, construct a sketch $\tilde{\tensor A} \in \R^{n \times n \times \ldots \times n}$ such that
\begin{equation}
\label{ineq::spectral norm bound}
\norm{\tensor A - \tilde{\tensor A}}_2 \leq \epsilon\norm{\tensor A}_2
\end{equation}
and the number of non-zero entries in $\tilde{\tensor A}$ is minimized. Here, the $\norm{\tensor A}_2$ norm is called the spectral norm of the tensor $\tensor A$ (see Section~\ref{eqt::tensor spectral norm definition} for the definition).
\end{definition}
\noindent A few comments are necessary to better understand the above definition. First, an order-$d$ tensor is simply a $d$-way array (obviously, a matrix is an order-2 tensor). We let $\norm{\cdot}_2$ denote the spectral norm of a tensor (see Section~\ref{sxn:notation} for notation), which is a natural extension of the matrix spectral norm. It is worth noting that exactly computing the tensor spectral norm is computationally hard. Second, a similar problem could be formulated by seeking a bound for the Frobenius norm of $\tensor A - \tilde{\tensor A}$. Third, this definition places no constraints on the form of the entries of $\tilde{\tensor A}$. However, in this work, we will focus on methods that return matrices and tensors $\tilde{\tensor A}$ whose entries are either zeros or (rescaled) entries of $\tensor A$. Prior work has investigated quantization as an alternative construction for the entries of $\tilde{\tensor A}$, while the theoretical properties of more general methods remain vastly unexplored. Fourth, the running time needed to construct a sketch is not restricted. All prior work has focused on the construction of sketches in one or two sequential passes over the input matrix or tensor. Thus, we are particularly interested in sketching algorithms that can be implemented within the same framework (a small number of sequential passes).

We conclude this section by discussing applications of the sparse sketches of Definition~\ref{def:matrixsparsification}. In the case of matrices, there are at least three important applications: approximate eigenvector computations, semi-definite programming (SDP) solvers, and matrix completion. The first two applications are based on the fact that, given a vector $x \in \R^n$, the product $\tensor A x$ can be approximated by $\tilde{\tensor A} x$ with a bounded loss in accuracy. The running time of the latter matrix-vector product is proportional to the number of non-zeros in $\tilde{\tensor A}$, thus leading to immediate computational savings. This fast matrix-vector product operation can then be used to approximate eigenvectors and eigenvalues of matrices~\cite{AM01,AM07,AHK06} via subspace iteration methods; yet another application would be a quick estimate of the Krylov subspace of a matrix. Additionally~\cite{AHK05,Asp09} argue that fast matrix-vector products are useful in SDP solvers. The third application domain of sparse sketches is the so-called matrix completion problem, an active research area of growing interest, where the user only has access to $\tilde{\tensor A}$ (typically formed by sampling a small number of elements of $\tensor A$ uniformly at random) and the goal is to reconstruct the entries of $\tensor A$ as accurately as possible. The motivation underlying the matrix completion problem stems from recommender systems and collaborative filtering and was initially discussed in~\cite{AFKMS01}. More recently, methods using bounds on $\tensor A-\tilde{\tensor A}$ and trace minimization algorithms have demonstrated exact reconstruction of $\tensor A$ under -- rather restrictive -- assumptions~\cite{CR09,CT09}. We expect that our work here will stimulate research towards generalizing matrix completion to tensor completion. More specifically, our tensor spectral norm bound could be a key ingredient in analyzing tensor completion algorithms, just like similar bounds for matrix sparsification were critical in matrix completion~\cite{CR09,CT09}. Finally, similar applications in recommendation systems, collaborative filtering, monitoring IP traffic patterns over time, etc. exist for the $d>2$ case in Definition~\ref{def:matrixsparsification}; see~\cite{Tso09,MMD06,MMD08} for details.

\subsection{Our algorithm and our main theorem}

\begin{algorithm}%[H]
\centerline{\caption{Tensor Sparsification Algorithm}}
\begin{algorithmic}[1]
\STATE \textbf{Input:} order-$d$ tensor $\tensor A \in \R^{n \times n \ldots \times n}$, sampling parameter $s$.

\textbf{For all} $i_1, ..., i_d \in [n]\times \ldots \times [n]$ \textbf{do}
\begin{itemize}
\item \textbf{If} $\tensor A_{i_1 ...i_d}^2 \leq \frac{\ln^d n}{n^{d/2}} \frac{\norm{\tensor A}_F^2}{s}$ \textbf{then}
    $$  \widetilde{\tensor A}_{i_1 ...i_d} = 0,
    $$
\item \textbf{ElseIf} $\tensor A_{i_1 ...i_d}^2 \geq \frac{\norm{\tensor A}_F^2}{s}$ \textbf{then}
    $$  \widetilde{\tensor A}_{i_1 ...i_d} = \tensor A_{i_1 ...i_d},
    $$
\item \textbf{Else}
    $$
    \widetilde{\tensor A}_{i_1 ...i_d} =
    \begin{cases}
    \frac{\tensor A_{i_1 ...i_d}}{p_{i_1 ...i_d}} & \text{,with probability     } p_{i_1 ...i_d} = \frac{s \tensor A^2_{i_1 ...i_d}}{\norm{\tensor A}^2_F}\\
    0 & \text{,with probability     } 1 - p_{i_1 ...i_d} \\
    \end{cases}
    $$
\end{itemize}
\STATE \textbf{Output:} Tensor $\widetilde{\tensor A} \in \R^{n \times n \ldots \times n}$.
\end{algorithmic}
\end{algorithm}

Our main algorithm (Algorithm 1) zeroes out ``small'' elements of the tensor $\tensor A$, keeps ``large'' elements of the tensor $\tensor A$, and randomly samples the remaining elements of the tensor $\tensor A$ with a probability that depends on their magnitude. The following theorem is our main quality-of-approximation result for Algorithm 1.

\begin{thm}
\label{thm::main theorem of tensor sparsification}
Let $\tensor A \in \R^{n \times...\times n}$ be an order-$d$ tensor and let $\widetilde{\tensor A}$ be constructed as described in Algorithm 1. Assume that $n \geq 320$. For $d \geq 3$, if the sampling parameter $s$ satisfies
\begin{equation}
\label{inq::upper bound of s with d >= 3}
s =\Omega\left( \frac{d^3 20^{2d} n^{d/2} \ln^d n}{\epsilon^2} \max \left\{ 1, \frac{\ln^{d+1} n}{n^{d/2-1}} \right\}  \norm{\tensor A}_F^2 \right),
\end{equation}
then, with probability at least $1 - n^{-2d}$,
$$
\norm{\tensor A - \widetilde{\tensor A}}_2 \leq \epsilon,
$$
where the tensor spectral norm $\norm{\cdot}_2$ is defined in (\ref{eqt::tensor spectral norm definition}). For $d=2$, the same spectral norm bound holds whenever the sampling parameter $s$ satisfies
\begin{equation}
\label{inq::upper bound of s with d = 2}
s =\Omega\left( \frac{ n \ln^5 n}{\epsilon^2} \norm{\tensor A}_F^2 \right).
\end{equation}
\end{thm}
\noindent The number of samples $s$ in Theorem \ref{thm::main theorem of tensor sparsification} involves the tensor Frobenius norm. In the following corollary, we restate the theorem by using the stable rank of a tensor, denoted by $\text{sr} \left(\tensor A\right)$. The stable rank of a tensor is defined analogously to the stable rank of a matrix, namely the ratio
$$\text{sr}\left(\tensor A\right) \triangleq \frac{\norm{\tensor A}_F^2}{\norm{\tensor A}_2^2}.$$
\begin{corol}
\label{cor::matrix sparsification}
Let $\tensor A \in \R^{n \times n}$ (assume $n \geq 320$) be an order-$d$ tensor and let $\widetilde{\tensor A}$ be constructed as described in Algorithm 1. If $n \geq \ln^8 n$ and the sampling parameter $s$ is set to
$$s = \Omega\left( \frac{d^2 20^{2d} n^{d/2} \ln^d n}{\epsilon^2} \text{sr} (\tensor A) \right),$$
then, with probability at least $1 - n^{-2d}$,
$$
\norm{\tensor A - \widetilde{\tensor A}}_2 \leq \epsilon \norm{\tensor A}_2.
$$
For $d=2$, the sampling parameter $s$ is simplified to  $s = \Omega\left( \frac{ n \ln^5 n}{\epsilon^2} \text{sr} (A) \right).$
\end{corol}
In both Theorem~\ref{thm::main theorem of tensor sparsification} and Corollary~\ref{cor::matrix sparsification}, $\tilde{\tensor A}$ has, in expectation, at most $2s$ non-zero entries and the construction of $\tilde{\tensor A}$ can be implemented in one pass over the input tensor/matrix $\tensor A$. Towards that end, we need to combine Algorithm~1 with the \textsc{Sample} algorithm presented in Section 4.1 of \cite{AM07}. Finally, in the context of Definition~\ref{def:matrixsparsification}, our result essentially shows that we can get a sparse sketch $\tilde{\tensor A}$ with $2s$ non-zero entries. In Theorem \ref{thm::main theorem of tensor sparsification} and Corollary~\ref{cor::matrix sparsification}, we have not made any attempt to optimize the constants which could potentially be reduced. In addition, when $n \geq \ln^8 n$, the maximum value in (\ref{inq::upper bound of s with d >= 3}) is at most one and the sampling parameter can be simplified to $s =\Omega\left( \frac{n^{d/2} \ln^d n}{\epsilon^2} \text{sr} (\tensor A) \right)$. Ignoring the poly$log$ factor, the theorem implies that out of the $n^d$ entries of the tensor, the algorithm only needs to selectively keep $\Omega(n^{d/2} \text{sr} (\tensor A))$ entries and zero out the rest, while accurately approximating the spectral norm of the original tensor.

Finally, we discuss our bound in light of the so-called Kruskal and Tucker rank of a tensor. Let $\text{kr}\left(\tensor A\right)$ be the Kruskal rank of the $d$-mode tensor $\tensor A$; see~\cite{Kolda2009} for the definition of the Kruskal rank and notice that the Kruskal rank is equal to the matrix rank when $d$ is equal to two. It is known that the number of degrees of freedom of a tensor is of the order $n \text{kr} \left(\tensor A\right)$. While, in general, the inequality $\text{sr}\left(\tensor A\right) \leq \text{kr}\left(\tensor A\right)$ does not hold, it does hold for the $d=2$ case as well as for some tensors that can be orthogonally decomposed~\cite{Kolda2001}. Another better way to bound the stable rank of a tensor is via the Tucker decomposition, which is similar to singular value decomposition of a matrix (see \cite{Kolda2009} for the definition). Decompose the order-$d$ tensor $\tensor A$ via
$$
\tensor A = \sum_{i_1=1}^{k_1} \cdots \sum_{i_d=1}^{k_d} g_{i_1 \cdots i_d} u_{i_1} \times_1 \cdots \times_d v_{i_d} = \tensor G \times_1 U \cdots \times V 
$$
where $U$,..., $V$ are orthogonal matrices of size $n \times k_1$, ..., $n \times k_d$, respectively; $\tensor G$ is the core tensor of size $k_1 \times \cdots \times k_d$. Here, the tensor-vector product is defined later in Section \ref{sxn:notation}. The tuple ($k_1,...,k_d$) is called the Tucker rank of the tensor $\tensor A$ where each $k_i$ is the column rank of the matrix $A_{(i)}$ constructed by unfolding $\tensor A$ along the $i$th direction. It can be easily seen that the degree of freedom of $\tensor A$ is roughly $n \sum_{i=1}^d k_i + \prod_{i=1}^d k_i$. In addition, the tensor Frobenius norm is 
$$
\norm{\tensor A}_F^2 = \norm{\tensor G}_F^2 \leq \left( \prod_{i=1}^d k_i \right) \max_{i_1,...,i_d} g^2_{i_1 \cdots i_d},
$$
and the spectral norm of $\tensor A$ (see Section \ref{sxn:notation} for the definition) is crudely lower bounded by $\max_{i_1,...,i_d} g_{i_1 \cdots i_d}$. Combining these two bounds and the fact that $\norm{\tensor A}_F \geq \norm{\tensor A}$ yield
$$
1 \leq \text{sr}\left(\tensor A\right) \leq  \prod_{i=1}^d k_i .
$$
In these situations, Corollary~\ref{cor::matrix sparsification} essentially implies that in order for the sampled tensor to be close to the original one, the number of samples required is at most on the order of $\Omega (n^{d/2} \prod_{i=1}^d k_i)$, which is proportional to $\Omega(n^{d/2})$ for low Tucker rank tensor. This bound is substantially larger than the tensor's degree of freedom $n \sum_{i=1}^d k_i + \prod_{i=1}^d k_i$. An open question is whether the $d/2$ power in the number of samples can be removed?

%An open question is whether the number of samples can be reduced to be proportional to the number of degrees of freedom of the tensor $n \sum_{i=1}^d k_i + \prod_{i=1}^d k_i$.

% proportional to the number of degrees of freedom of the tensor.

\subsection{Comparison with prior work}

To the best of our knowledge, for $d>2$, there exists no prior work on element-wise tensor sparsification that provides results comparable to Theorem~\ref{thm::main theorem of tensor sparsification}. It is worth noting that the work of~\cite{Tso09} deals with the Frobenius norm of the tensor, which is much easier to manipulate, and its main theorem is focused on approximating the so-called HOSVD of a tensor, as opposed to decomposing the tensor as a sum of rank-one components.

For the $d=2$ case, prior work does exist and we will briefly compare our results in Corollary~\ref{cor::matrix sparsification} with current state-of-the-art. In summary, our result in Corollary~\ref{cor::matrix sparsification} outperforms prior work, in the sense that, using the same accuracy parameter $\epsilon$ in Definition~\ref{def:matrixsparsification}, the resulting matrix $\tilde{\tensor A}$ has fewer non-zero elements. In \cite{AM01,AM07} the authors presented a sampling method that requires at least $O(\textbf{st}\left(\tensor A\right) n \ln^4 n/\epsilon^2)$ non-zero entries in $\tilde{\tensor A}$ in order to achieve the proposed accuracy guarantee. (Here $\textbf{st}\left(\tensor A\right)$ denotes the stable rank of the matrix $\tensor A$ that is always upper bounded by the rank of $\tensor A$.) Our result increases the sampling complexity by a $\ln n$ factor. This increment is due to the more general model (tensor) we consider. In \cite{BSS09, SS08} the authors proposed sparsification schemes for structural Laplacian matrix and thus required smaller amount of non-zero entries, while our method can apply for any matrix $\tensor A$ with no restriction on its structure. It is harder to compare our method to the work of~\cite{AHK06}, which depends on the $\sum_{i,j=1}^n \abs{\tensor A_{ij}}$. The latter quantity is, in general, upper bounded only by $n \norm{\tensor A}_F$, in which case the sampling complexity of~\cite{AHK06} is much worse, namely $O(\textbf{st}\left(\tensor A\right) n^{3/2}/\epsilon)$. However, it is worth noting that the result of~\cite{AHK06} is appropriate for matrices whose ``energy'' is focused only on a small number of entries, as well as that their bound holds with much higher probability than ours.

In parallel with our work, two related results appeared in ArXiv. First,~\cite{GT09} studied the $\norm{\cdot}_{\infty \rightarrow 2}$ and $\norm{\cdot}_{\infty \rightarrow 1}$ norms in the matrix sparsification context. The authors also presented a sampling scheme for the problem of Definition~\ref{def:matrixsparsification}. Additionally,~\cite{DZ09} leveraged a powerful matrix Bernstein inequality and improved the sampling complexity of Corollary~\ref{cor::matrix sparsification} by an $O (\ln^2 n)$ factor. Subsequently to our work,~\cite{AKL13} presented an alternative approach to~\cite{DZ09} that is based on $\ell_1$ sampling, e.g., sampling with respect to the absolute values of the entries of a matrix as opposed to their squares. However, neither of the aforementioned results generalizes to tensors. Indeed, establishing analogous bounds for $d$-mode tensors is a major open problem.

\subsection{Bounding the spectral norm of random tensors}

An important contribution of our work is the technical analysis and, in particular, the proof of a bound for the spectral norm of random tensors that is necessary in order to prove Theorem~\ref{thm::main theorem of tensor sparsification}. It is worth noting that all known results for the $d=2$ case of Theorem~\ref{thm::main theorem of tensor sparsification} are either combinatorial in nature (e.g., the proofs of~\cite{AM01,AM07} are based on the result of~\cite{FK81}, whose proof is fundamentally combinatorial) or use simple $\epsilon$-net arguments~\cite{AHK06}. The only exceptions are the recent results in~\cite{DZ09,GT09} which leverage powerful Bernstein and Chernoff-type inequalities for matrices \cite{Tropp2012friendly}. It is also important to emphasize that over the last few years, there are active research in establish sharp bound for the sum of random matrices \cite{AW2002,Oli2010,Tropp2012friendly} (see the tutorial paper \cite{Tropp2015tutorial} of Tropp for more references). As stated above, none of these approaches can be extended to the $d > 2$ case; indeed, the $d>2$ case seems to require novel tools and methods. In our work, we are only able to prove the following theorem using the so-called \textit{entropy-concentration tradeoff}, an analysis technique that was originally developed by Latala \cite{Latala_2005_RandomMatrix_journal} and has been recently investigated by Mark Rudelson and Roman Vershynin~\cite{Rudelson_Vershynin_2008_RandomRecMatrix_journal,Vershynin_SpectralNormBA_2009}. The following theorem presents a spectral norm bound for random tensors and is fundamental in proving Theorem~\ref{thm::main theorem of tensor sparsification}.

\begin{thm}
\label{thm::bound tensor (E norm(B) ^q)^(1/q)}
Let $\widehat{\tensor A} \in \R^{n \times ...\times n}$ be an order-$d$ tensor and let $\tensor A$ be a random tensor of the same dimensions whose entries are independent and $\E \tensor A = \widehat{\tensor A}$. For any $\lambda \leq \frac{1}{64}$, assume that $1 \leq q \leq 2d \lambda n \ln \frac{5e}{\lambda}$. Then,
\begin{equation*}
\begin{split}
\left( \E \norm{\tensor A - \widehat{\tensor A}}_2^q \right)^{\frac{1}{q}} \leq & c 8^d \sqrt{2d \ln \left(\frac{5e}{\lambda} \right)} \left( \left[ \log_2 \left( \frac{1}{\lambda} \right) \right]^{d-1} \left( \sum_{j=1}^d \E_{\tensor A} \alpha_j^q \right)^{\frac{1}{q}} + \sqrt{\lambda n} \left( \E_{\tensor A} \beta^q \right)^{\frac{1}{q}} \right),
\end{split}
\end{equation*}
where
$$
\alpha_j^2 \triangleq \max_{i_1,\ldots,i_{j-1},i_{j+1},\ldots,i_d} \left(\sum_{i_j=1}^n \tensor A_{i_1 \ldots i_{j-1} i_j i_{j+1} ... i_d}^2\right) \quad \text{and} \quad \beta = \max_{i_1,...,i_d} |\tensor A_{i_1...i_d}|.
$$
In the above inequality, $c$ is a small constant and $\norm{\cdot}_2$ refers to the tensor spectral norm defined in Section~\ref{eqt::tensor spectral norm definition}.
\end{thm}

\noindent An immediate corollary of the above theorem emerges by setting tensor $\widehat{\tensor A}$ to zero.
\begin{corol}\label{cor::bound tensor (E norm(B) ^q)^(1/q)}
Let $\tensor B \in \R^{n \times ... \times n}$ be a random order-$d$ tensor, whose entries are independent, zero-mean, random variables. For any $\lambda \leq \frac{1}{64}$, assume that $1 \leq q \leq 2d \lambda n \ln \frac{5e}{\lambda}$. Then,
\begin{equation*}
\left( \E \norm{\tensor B}_2^q \right)^{\frac{1}{q}} \leq c 8^d \sqrt{2d \ln \left(\frac{5e}{\lambda} \right)} \left( \left[ \log_2 \left( \frac{1}{\lambda} \right) \right]^{d-1} \left( \sum_{j=1}^d \E_{\tensor B} \alpha_j^q \right)^{\frac{1}{q}} + \sqrt{\lambda n} \left( \E_{\tensor B} \beta^q \right)^{\frac{1}{q}} \right),
\end{equation*}
where
$$
\alpha_j^2 \triangleq \max_{i_1,\ldots,i_{j-1},i_{j+1},\ldots,i_d} \left(\sum_{i_j=1}^n \tensor B_{i_1 \ldots i_{j-1} i_j i_{j+1} ... i_d}^2\right) \quad \text{and} \quad \beta = \max_{i_1,...,i_d} |\tensor B_{i_1...i_d}|.
$$
In the above inequality, $c$ is a small constant and $\norm{\cdot}_2$ refers to the tensor spectral norm defined in Section~\ref{eqt::tensor spectral norm definition}.
\end{corol}

\noindent As will be clear in the proof, the parameter $\lambda$ defines the entropy-concentration tradeoff. Depending on particular properties of the random tensor $\tensor B$, one can set the parameter $\lambda$ so that the bound on the right-hand side is optimized. In particular, when the entries of $\tensor B$ are of similar magnitudes (formally, $\max_j \alpha^2_j = c_1 n \beta^2$), we can choose $\lambda$ to be a small constant. (Note that we always have $\max_j \alpha^2_j \leq n \beta^2$.) In this case, we have a simplified result.

\begin{corol}\label{cor::bound tensor (E norm(B) ^q)^(1/q) - fixed lambda}
Let $\tensor B \in \R^{n \times ... \times n}$ be a random order-$d$ tensor, whose entries are independent, zero-mean, random variables. Assume that $1 \leq q \leq C d n$. Also, assume that $ c_1 n \beta^2 \leq \max_j \alpha^2_j \leq C_1 n \beta^2$. Then,
\begin{equation*}
\left( \E \norm{\tensor B}_2^q \right)^{\frac{1}{q}} \leq c_d 8^d \sqrt{d}  \left( \sum_{j=1}^d \E_{\tensor B} \max_{i_1,\ldots,i_{j-1},i_{j+1},\ldots,i_d} \left(\sum_{i_j=1}^n \tensor B_{i_1 \ldots i_{j-1} i_j i_{j+1} ... i_d}^2\right)^{\frac{q}{2}} \right)^{\frac{1}{q}}.
\end{equation*}
In the above inequality, $c_d$ is a small constant depending on $d$ and $\norm{\cdot}_2$ refers to the tensor spectral norm defined in Section~\ref{eqt::tensor spectral norm definition}.
\end{corol}

\noindent We note that this bound is optimal since $\norm{\tensor B}_2$ is always lower bounded by
$$\max_{i_1,\ldots,i_{j-1},i_{j+1},\ldots,i_d} \left(\sum_{i_j=1}^n \tensor B_{i_1 \ldots i_{j-1} i_j i_{j+1} ... i_d}^2\right)^{\frac{1}{2}} .$$ 
We also note that for the matrix case ($d=2$), the result of Corollary~\ref{cor::bound tensor (E norm(B) ^q)^(1/q) - fixed lambda} has a very similar structure with the result of~\cite{Latala_2005_RandomMatrix_journal}. In fact, our proof strategy is borrowed from~\cite{Latala_2005_RandomMatrix_journal}, with significant modifications in order to adapt it to higher-order tensors. For a general random tensor, we can use the crude bound $\beta \leq \max_j \alpha_j$ and also set $\lambda = \frac{(\ln n)^{2(d-1)}}{n}$. Then, the following corollary provides a bound for the spectral norm of the random tensor.

\begin{corol}\label{cor::bound tensor (E norm(B) ^q)^(1/q) - fixed lambda 2nd}
Let $\tensor B \in \R^{n \times ... \times n}$ be a random order-$d$ tensor, whose entries are independent, zero-mean, random variables. Assume that $1 \leq q \leq C d \ln n$. Then,
\begin{equation*}
\left( \E \norm{\tensor B}_2^q \right)^{\frac{1}{q}} \leq c_d 8^d \left( \ln n \right)^{d-1/2} \left( \sum_{j=1}^d \E_{\tensor B} \max_{i_1,\ldots,i_{j-1},i_{j+1},\ldots,i_d} \left(\sum_{i_j=1}^n \tensor B_{i_1 \ldots i_{j-1} i_j i_{j+1} ... i_d}^2\right)^{\frac{q}{2}} \right)^{\frac{1}{q}}.
\end{equation*}
In the above inequality, $c_d$ is a small constant depending on $d$ and $\norm{\cdot}_2$ refers to the tensor spectral norm defined in Section~\ref{eqt::tensor spectral norm definition}.
\end{corol}

\section{Preliminaries}

\subsection{Notation}\label{sxn:notation}

We will use $[n]$ to denote the set $\left\{1,2,\ldots ,n\right\}$. $c_0$, $c_1$, $c_2$, etc. will denote small numerical constants, whose values change from one section to the next. $\E X$ will denote the expectation of a random variable $X$. When $X$ is a matrix, then $\E X$ denotes the element-wise expectation of each entry of $X$. Similarly, $\V\left(X\right)$ denotes the variance of the random variable $X$ and $\Prob \left({\cal E}\right)$ denotes the probability of event ${\cal E}$. Finally, $\ln x$ denotes the natural logarithm of $x$ and $\log_2 x$ denotes the base two logarithm of $x$.

We briefly remind the reader of vector norm definitions. Given a vector $x \in \R^n$ the $\ell_2$ norm of $x$ is denoted by $\norm{x}_2$ and is equal to the square root of the sum of the squares of the elements of $x$. Also, the $\ell_0$ norm of the vector $x$ is equal to the number of non-zero elements in $x$. Finally, given a Lipschitz function $f: \R^n \mapsto \R$ we define the Lipschitz norm of $f$ to be
$$
\norm{f}_{L} = \sup_{x, y \in \R^n} \frac{\abs{ f(x) - f(y) }} {\norm{x-y}_2}.
$$
For any $d$-mode or order-$d$ tensor $\tensor A \in \R^{n \times \ldots \times n}$, its Frobenius norm $\norm{\tensor A}_F$ is defined as the square root of the sum of the squares of its elements. We now define tensor-vector products as follows: let $x, y$ be vectors in $\R^n$. Then,
\begin{eqnarray*}
\tensor A \times_{1} x &=& \sum_{i=1}^n \tensor A_{i j k \ldots \ell} x_i,\\
\tensor A \times_{2} x &=& \sum_{j=1}^n \tensor A_{i j k \ldots \ell} x_j,\\
\tensor A \times_{3} x &=& \sum_{k=1}^n \tensor A_{i j k \ldots \ell} x_k, \text{ etc. }
\end{eqnarray*}
Note that the outcome of the above operations is an order-$(d-1)$ tensor. The above definition may be extended to handle multiple tensor-vector products, e.g.,
$$ \tensor A \times_1 x \times_2 y  = \sum_{i=1}^n \sum_{j=1}^n \tensor A_{ijk \ldots \ell} x_{i} y_{j}.$$
Note that the outcome of the above operation is an order-$(d-2)$ tensor. Using this definition, the spectral norm of a tensor is defined as
\begin{equation}
\label{eqt::tensor spectral norm definition}
\norm{\tensor A}_2 = \sup_{x_1 \ldots x_d \in \Sphere^n} \abs{\tensor A \times_1 x_1 \ldots \times_d x_d},
\end{equation}
where $\Sphere^n$ is the unit sphere in $n$-dimensional space. In words, the vectors $x_i \in \R^n$ are unit vectors, i.e., $\norm{x_i}_2=1$ for all $i \in [d]$. It is worth noting that $\tensor A \times_1 x_1 \ldots \times_d x_d \in \R$ and also that our tensor norm definitions when restricted to matrices (order-2 tensors) coincide with the standard definitions of matrix norms.

We also present an inequality that will be useful in our work. For any two $d$-mode tensors $\tensor A$ and $\tensor B$ of the same dimensions and any scalar $q \geq 1$,
\begin{equation}
\label{inq::bound sum of tensor spectral norm}
\norm{\tensor A + \tensor B}_2^q \leq 2^{q-1} (\norm{\tensor A}_2^q + \norm{\tensor B}_2^q).
\end{equation}
The proof is quite simple. Notice that for nonnegative scalars $x$ and $y$, $(x+y)^q \leq 2^{q-1} \left(x^q + y^q\right)$ for $q \geq 1$ (see Lemma~\ref{lem::bound sum of power of q} for a more general proof). Thus, for any $x_1,...,x_d \in \Sphere^n$,
$$
\abs{\tensor A \times_1 x_1 \ldots \times_d x_d + \tensor B \times_1 x_1 \ldots \times_d x_d}^q \leq 2^{q-1} \abs{\tensor A \times_1 x_1 \ldots \times_d x_d}^q + 2^{q-1} \abs{\tensor B \times_1 x_1 \ldots \times_d x_d}^q.
$$
Taking the maximum of both sides completes the proof.

\subsection{Measure concentration}

We will need the following version of Bennett's inequality.
\begin{lem}
\label{thm::Bennett's inequality}
Let $X_1$, $X_2$,..., $X_n$ be independent, zero-mean, random variables with $\abs{X_i} \leq 1$.
For any $t \geq \frac{3}{2}\sum_{i=1}^n \V (X_i) >0$
$$
\Prob \left( \sum_{i=1}^n X_i > t  \right) \leq e^{- t/2 }.
$$
\end{lem}
\noindent This version of Bennett's inequality can be derived from the standard one, stating that
$$\Prob \left( \sum_{i=1}^n X_i > t  \right) \leq e^{- \sigma^2 h \left( t / \sigma^2 \right)}.$$
Here $\sigma^2 = \sum_{i=1}^n \V (X_i)$ and $h(u) = (1+u) \ln (1+u) - u$. Lemma~\ref{thm::Bennett's inequality} follows using the fact that $h(u) \geq u/2$ for $u \geq 3/2$.
We also remind the reader of the following well-known result on measure concentration (see, for example, eqn. (1.4) of \cite{Ledoux_Talagrand_1991_ProbBanachSpace_book}).

\begin{lem} \label{thm::Gaussian concentration}
Let $f: \R^n \mapsto \R$ be a Lipschitz function and let $\norm{f}_{L}$ be its Lipschitz norm. If $g \in \R^n$ is a standard Gaussian vector (i.e., a vector whose entries are independent standard Gaussian random variables), then for all $t > 0$
$$
\Prob \left( f (g) \geq \E f(g) +  t\sqrt{2} \norm{f}_{L} \right) \leq e^{-t^2}.
$$
\end{lem}

\noindent The following lemma, whose proof may be found in the Appendix, converts a probabilistic bound for the random variable $X$ to an expectation bound for $X^q$, for all $q \geq 1$, and might be of independent interest.

\begin{lem} \label{lem::convert from probabilistic bound to expectation bound}
Let $X$ be a random variable assuming non-negative values. For all $t\geq 0$ and non-negative $a$, $b$, and $h$:

\noindent \textbf{(a)} If
$
\Prob\left(X \geq a + t b\right) \leq e^{-t + h},
$
then, for all $q \geq 1$,
$$
\E X^q \leq 2 (a + bh + bq)^q.
$$

\noindent \textbf{(b)} If
$
\Prob(X \geq a + t b) \leq e^{-t^2 + h},
$
then, for all $q \geq 1$,
$$
\E X^q \leq 3\sqrt{q}\left(a + b\sqrt{h} + b\sqrt{q/2}\right)^q.
$$
\end{lem}

\noindent Finally, we present an $\epsilon$-net argument that we will repeatedly use. Recall from Lemma $3.18$ of~\cite{Ledoux_2001_ConcentrationMeasure_book} that the cardinality of an $\epsilon$-net on the unit sphere is at most $\left(1 + 2/\epsilon\right)^n$. The following lemma essentially generalizes the results of Lecture 6 of~\cite{Vershynin_Course} to order-$d$ tensors.
%
%\begin{lem}\label{prop::dicretize of norm of a tensor A}
%Let $\N$ be an $\epsilon$-net for the unit sphere $\Sphere^{n-1}$ in $\R^n$. Then, the spectral norm of a $d$-mode tensor $\tensor A$ is bounded by
%$$
%\norm{\tensor A}_2 \leq \left( \frac{1}{1 - \epsilon} \right)^{d-1} \sup_{x_1 \ldots x_{d-1} \in \N} \norm{\tensor A \times_1 x_1 \ldots \times_{d-1} x_{d-1}}_2.
%$$
%\end{lem}
%

\begin{lem}\label{prop::dicretize of norm of a tensor A}
Let $\N$ be an $\epsilon$-net for a set $B$ associated with a norm $\norm{\cdot}$. Then, the spectral norm of a $d$-mode tensor $\tensor A$ is bounded by
$$
\sup_{x_1 \ldots x_{d-1} \in B} \norm{\tensor A \times_1 x_1 \ldots \times_{d-1} x_{d-1}}_2 \leq \left( \frac{1}{1 - \epsilon} \right)^{d-1} \sup_{x_1 \ldots x_{d-1} \in \N} \norm{\tensor A \times_1 x_1 \ldots \times_{d-1} x_{d-1}}_2.
$$
\end{lem}

\noindent Notice that, using our notation, $\tensor A \times_1 x_1 \ldots \times_{d-1} x_{d-1}$ is a vector in $\R^n$. The proof of the lemma may be found in the Appendix. An immediate implication of our result is that the spectral norm of a $d$-mode tensor $\tensor A$ is bounded by
$$
\norm{\tensor A}_2 \leq \left( \frac{1}{1 - \epsilon} \right)^{d-1} \sup_{x_1 \ldots x_{d-1} \in \N} \norm{\tensor A \times_1 x_1 \ldots \times_{d-1} x_{d-1}}_2,
$$
where $\N$ is the $\epsilon$-net for the unit sphere $\Sphere^{n-1}$ in $\R^n$.

\section{Bounding the spectral norm of random tensors}

This section will focus on proving Theorem~\ref{thm::bound tensor (E norm(B) ^q)^(1/q)}, which essentially bounds the spectral norm of random tensors. Towards that end, we will first apply a symmetrization argument following the lines of~\cite{Latala_2005_RandomMatrix_journal}. This argument will allow us to reduce the task-at-hand to bounding the spectral norm of a Gaussian random tensor. As a result, we will develop such an inequality by employing the so-called entropy-concentration technique, which has been developed by Mark Rudelson and Roman Vershynin~\cite{Rudelson_Vershynin_2008_RandomRecMatrix_journal,Vershynin_SpectralNormBA_2009}.

For simplicity of exposition and to avoid carrying multiple indices, we will focus on proving Theorem~\ref{thm::bound tensor (E norm(B) ^q)^(1/q)} for order-3 tensors (i.e., $d=3$). Throughout the proof, we will carefully comment on derivations where $d$ (the number of modes of the tensor) affects the bounds of the intermediate results. Notice that if $d=3$, then a tensor $\tensor A \in \R^{n \times n \times n}$ may be expressed as
\begin{equation}\label{eqn:tensorsum}
\tensor A = \sum_{i,j,k=1}^n \tensor A_{ijk} \cdot e_i \otimes e_j \otimes e_k.
\end{equation}
In the above, the vectors $e_i \in \R^n$ (for all $i \in [n]$) denote the standard basis for $\R^n$ and $\otimes$ denotes the outer product operation. Thus, for example, $e_i \otimes e_j \otimes e_k$ denotes an tensor in $\mathbb{R}^{n \times n \times n}$ whose $(i,j,k)$-th entry is equal to one, while all other entries are equal to zero.

\subsection{A Gaussian symmetrization inequality}

The main result of this section can be summarized in Lemma~\ref{lem:gaussiansymmetrization}. In words, the lemma states that, by losing a factor of $\sqrt{2 \pi}$, we can independently randomize each entry of $\tensor A$ via a Gaussian random variable. Thus, we essentially reduce the problem of finding a bound for the spectral norm of a tensor $\tensor A$ to finding a bound for the spectral norm of a Gaussian random tensor.
\begin{lem}\label{lem:gaussiansymmetrization}
Let $\widehat{\tensor A} \in \R^{n \times n\times n}$ be any order-3 tensor and let $\tensor A$ be a random tensor of independent entries and of the same dimensions such that $\E_{\tensor A} \tensor A = \widehat{\tensor A}$. Also let the $g_{ijk}$ be Gaussian random variables for all triples $\left(i,j,k\right) \in [n]\times [n] \times [n]$. Then for any $q \geq 1$,
\begin{equation} \label{ineq::bound expectation of (pi_omega - pI) (T) - step 4}
\E_{\tensor A} \norm{ \tensor A - \widehat{\tensor A}}_2^q \leq \left(\sqrt{2\pi}\right)^q
\E_{\tensor A} \E_{g}\norm{\sum_{i,j,k} g_{ijk} \tensor A_{ijk} \cdot  e_i \otimes e_j \otimes e_k}^q_2.
\end{equation}
\end{lem}

\begin{proof}
Let $\tensor A'$ be an independent copy of the tensor $\tensor A$. By applying a symmetrization argument and Jensen's inequality, we get
$$
\E_{\tensor A} \norm{ \tensor A - \widehat{\tensor A}}_2^q = \E_{\tensor A} \norm{ \tensor A - \E_{\tensor A} \tensor A}_2^q = \E_{\tensor A} \norm{\tensor A - \E_{\tensor A'} \tensor A'}_2^q \leq \E_{\tensor A} E_{\tensor A'} \norm{\tensor A - \tensor A'}_2^q.
$$
Note that the entries of the tensor $\tensor A - \tensor A'$ are independent symmetric random variables and thus their distribution is the same as the distribution of the random variables $\epsilon_{ijk} \left(\tensor A_{ijk} - \tensor A'_{ijk}\right)$, where the $\epsilon_{ijk}$'s are independent, symmetric, Bernoulli random variables assuming the values $+1$ and $-1$ with equal probability. Hence,
\begin{eqnarray*}
\E_{\tensor A} \E_{\tensor A'} \norm{ \tensor A - \tensor A' }_2^q &=& \E_{\tensor A} \E_{\tensor A'} \E_{\epsilon} \norm{\sum_{i,j,k}\epsilon_{ijk} \left(\tensor A_{ijk} - \tensor A'_{ijk}\right)e_i \otimes e_j \otimes e_k}_2^q \\
&\leq& 2^{q-1} \E_{\tensor A} \E_{\epsilon} \norm{\sum_{i,j,k}\epsilon_{ijk} \tensor A_{ijk} e_i \otimes e_j \otimes e_k}_2^q\\
&+& 2^{q-1} \E_{\tensor A'} \E_{\epsilon} \norm{ \sum_{i,j,k}\epsilon_{ijk}\tensor A'_{ijk}e_i \otimes e_j \otimes e_k }_2^q.
\end{eqnarray*}
Here the inequality follows from eqn.~(\ref{inq::bound sum of tensor spectral norm}). Now, since the entries of the tensors $\tensor A$ and $\tensor A'$ have the same distribution, we get
\begin{equation}
\label{ineq::bound expectation of C - step 1}
\E_{\tensor A} \E_{\tensor A'} \norm{ \tensor A - \tensor A' }_2^q \leq 2^{q}  \E_{\tensor A} \E_{\epsilon} \norm{\sum_{i,j,k}\epsilon_{ijk} \tensor A_{ijk} e_i \otimes e_j \otimes e_k}_2^q.
\end{equation}
We now proceed with the Gaussian symmetrization argument. Let $g_{ijk}$ for all $i,j$, and $k$ be independent Gaussian random variables. It is well-known that $\E \abs{g_{ijk}} = \sqrt{2/\pi}$ . Using Jensen's inequality, we get
\begin{eqnarray*} \label{ineq::bound expectation of (pi_omega - pI) (T) - step 3}
\E_{\tensor A} \E_{\epsilon} \norm{\sum_{i,j,k}\epsilon_{ijk} \tensor A_{ijk} e_i \otimes e_j \otimes e_k}_2^q &=& \left(\frac{\pi}{2} \right)^{q/2}
\E_{\tensor A} \E_{\epsilon}\norm{\sum_{i,j,k}\epsilon_{ijk} \tensor A_{ijk} \left(\E_{g}\abs{g_{ijk}}\right) \cdot  e_i \otimes e_j \otimes e_k}_2^q\\
&\leq& \left(\frac{\pi}{2} \right)^{q/2}
\E_{\tensor A} \E_{\epsilon} \E_{g}\norm{\sum_{i,j,k}\epsilon_{ijk} \tensor A_{ijk} \abs{g_{ijk}} \cdot  e_i \otimes e_j \otimes e_k}_2^q\\
&=& \left(\frac{\pi}{2} \right)^{q/2}
\E_{\tensor A} \E_{g}\norm{\sum_{i,j,k} g_{ijk} \tensor A_{ijk} \cdot  e_i \otimes e_j \otimes e_k}_2^q.
\end{eqnarray*}
The last equality holds since $\epsilon_{ijk} \abs{g_{ijk}}$ and $g_{ijk}$ have the same distribution. Thus, combining the above with eqn.~(\ref{ineq::bound expectation of C - step 1}) we have finally obtained the Gaussian symmetrization inequality.
\end{proof} 
\subsection{Bounding the spectral norm of a Gaussian random tensor}\label{sxn:GaussianRandomTensor}

In this section we will seek a bound for the spectral norm of the tensor $\tensor H$ whose entries $\tensor H_{ijk}$ are equal to $g_{ijk} \tensor A_{ijk}$ (we are using the notation of Lemma~\ref{lem:gaussiansymmetrization}). Obviously, the entries of $\tensor H$ are independent, zero-mean Gaussian random variables. We would like to estimate
$$
\E_g \norm{\tensor H}^q = \E_g  \sup_{x, y} \norm{ \tensor H \times_1 x \times_2  y }^q_2
$$
over all unit vectors $x,y \in \R^n$. Our first lemma computes the expectation of the quantity $\norm{ \tensor H \times_1 x \times_2  y }_2$ for a fixed pair of unit vectors $x$ and $y$.
\begin{lem}
\label{lem::bound expected norm of B times x times y}
Given a pair of unit vectors $x$ and $y$
$$
\E_g \norm{\tensor H \times_1 x \times_2  y }_2 \leq \sqrt{\max_{i,j} \sum_k \tensor A_{ijk}^2}.
$$
\end{lem}

\begin{proof}
Let $s = \tensor H \times_1 x \times_2  y \in \R^n$ and let $s_k = \sum_{i,j} \tensor H_{ijk} x_i y_j$ for all $k \in [n]$. Thus,
\begin{eqnarray*}
\norm{s}_2^2 &=& \sum_{k} \left( \sum_{i,j} \tensor H_{ijk} x_i y_j \right)^2 \\
&=& \sum_{i,j,k} \tensor H_{ijk}^2 x^2_i y^2_j + \sum_k \sum_{i,j \neq p,q} \tensor H_{ijk} \tensor H_{pqk} x_i y_j x_p y_q.
\end{eqnarray*}
\noindent Using $\E_g \tensor H_{ijk} = 0$ and $\E_g \tensor H^2_{ijk} = \tensor A^2_{ijk} \E_g g^2_{ijk} = \tensor A^2_{ijk}$ we conclude that
$$
\E_g \norm{s}_2^2 = \sum_{i,j,k} \tensor A_{ijk}^2 x^2_i y^2_j = \sum_{i} x^2_i \sum_j y^2_j \sum_k \tensor A^2_{ijk}  \leq  \max_{i,j} \sum_k \tensor A_{ijk}^2.
$$
The last inequality follows since $\norm{x}_2 = \norm{y}_2 = 1$. Using $\E_g \norm{s}_2 \leq \sqrt{\E_g\norm{s}_2^2}$ we obtain the claim of the lemma.
\end{proof}

\noindent The next lemma argues that $\norm{\tensor H \times_1 x \times_2 y}_2$ is concentrated around its mean (which we just computed) with high probability.
\begin{lem}
\label{lem::bound norm of B times x times y in probability}
Given a pair of unit vectors $x$ and $y$
\begin{equation}
\label{ineq::bound norm of B times x times y in probability for a fixed set x and y}
\Prob \left( \norm{\tensor H \times_1 x \times_2 y}_2  \geq \sqrt{\max_{i,j} \sum_k \tensor A_{ijk}^2} + t \sqrt{2} \max_k \sqrt{\sum_{i,j} \tensor A^2_{ijk} x^2_i y^2_j } \right) \leq e^{-t^2}.
\end{equation}
\end{lem}

\begin{proof}
Consider the vector $s = \tensor H \times_1 x \times_2  y \in \R^n$ and recall that $\tensor H_{ijk} = g_{ijk} \tensor A_{ijk}$ to get
\begin{eqnarray*}
s &=& \sum_{i,j,k} \left(\tensor H_{ijk} x_i y_j\right)  e_k  \\
&=& \sum_k \left(\sum_{i,j} \tensor H_{ijk} x_i y_j\right)  e_k \\
&=& \sum_k \left(\sum_{i,j} g_{ijk} \tensor A_{ijk} x_i y_j\right) e_k.
\end{eqnarray*}
In the above the $e_k$ for all $k \in [n]$ are the standard basis vectors for $\R^n$. Now observe that all $g_{ijk} \tensor A_{ijk} x_i y_j$ are Gaussian random variables, which implies that their sum (over all $i$ and $j$) is also a Gaussian random variable with zero mean and variance $\sum_{i,j} \tensor A^2_{ijk} x^2_i y^2_j$. Let $$q_k^2 = \sum_{i,j} \tensor A^2_{ijk} x^2_i y^2_j \qquad \mbox{for all } k \in [n]$$ and rewrite the vector $s$ as the sum of weighted standard Gaussian random variables:
$$
s = \sum_k  z_k  q_k e_k.
$$
In the above the $z_k$'s are standard Gaussian random variables for all $k \in [n]$. Let $z$ be the vector in $\R^n$ whose entries are the $z_k$'s and let $$f(z) = \norm{\sum_k  z_k  q_k e_k}_2.$$ We apply Lemma~\ref{thm::Gaussian concentration} to $f(z)$. It is clear that
$
f^2(z) = \sum_k z^2_k q^2_k \leq  \norm{z}_2^2 \max_k q^2_k.
$
Therefore, the Lipschitz norm of $f$ is 
$$\norm{f}_{L} = \max_k \abs{q_k} = \max_k \left(\sum_{i,j} \tensor A^2_{ijk} x^2_i y^2_j \right)^{1/2} .$$
Applying Lemma \ref{thm::Gaussian concentration} and Lemma \ref{lem::bound expected norm of B times x times y} completes the proof.
\end{proof}

\subsubsection{An $\epsilon$-net construction: the entropy-concentration tradeoff argument}\label{sxn:ECintro}

Given the measure concentration result of Lemma~\ref{lem::bound norm of B times x times y in probability}, one might be tempted to bound the quantity $\norm{\tensor H \times_1 x \times_2 y}_2$ for all unit vectors $x$ and $y$ by directly constructing an $\epsilon$-net $N$ on the unit sphere. Since the cardinality of $N$ is well-known to be upper bounded by $\left(1+\frac{2}{\epsilon}\right)^n$, it follows that by getting an estimate for the quantity $\norm{\tensor H \times_1 x \times_2 y}$ for a pair of vectors $x$ and $y$ in $N$ and subsequently applying the union bound combined with Lemma~\ref{prop::dicretize of norm of a tensor A}, an upper bound for the norm of the tensor $\tensor H$ may be derived. Unfortunately, this simple technique does not yield a useful result: the failure probability of Lemma~\ref{lem::bound norm of B times x times y in probability} is not sufficiently small in order to permit the application of a union bound over all vectors $x$ and $y$ in $N$.

In order to overcome this obstacle, we will apply a powerful and novel argument, the so-called \textit{entropy-concentration tradeoff}, which was originally investigated by Latala \cite{Latala_2005_RandomMatrix_journal} and has been recently developed by Mark Rudelson and Roman Vershynin~\cite{Rudelson_Vershynin_2008_RandomRecMatrix_journal,Vershynin_SpectralNormBA_2009}. To begin with, we express a unit vector $x \in \R^n$ as a sum of two vectors $z,w \in \R^n$ satisfying certain bounds on the magnitude of their coordinates. Thus, $x = z + w$, where, for all $i \in [n]$,
\begin{eqnarray*}
z_i &=& \begin{cases}
        x_i & \text{if     } \abs{x_i} \geq \frac{1}{\sqrt{\lambda n}}\\
        0 & \text{,otherwise}
        \end{cases}\\
w_i &=& \begin{cases}
        x_i & \text{if     } \abs{x_i} < \frac{1}{\sqrt{\lambda n}}\\
        0 & \text{,otherwise}
        \end{cases}\\
\end{eqnarray*}
In the above $\lambda \in (0,1]$ is a small constant that will be specified later. It is easy to see that $\norm{z}_2 \leq 1$, $\norm{w}_2 \leq 1$, and that the number of non-zeros entries in $z$ (i.e., the $\ell_0$ norm of $z$) is bounded:
$$
\norm{z}_0 \leq \lambda n.
$$
Essentially, we have ``split'' the entries of $x$ in two vectors: a sparse vector $z$ with a bounded number of non-zero entries and a spread vector $w$ with entries whose magnitude is restricted. Thus,  we can now divide the unit sphere into two sets:
\begin{eqnarray*}
B_{2,0} &=& \left\{x \in \R^n: \norm{ x}_2 \leq 1,  \abs{x_i} \geq \frac{1}{\sqrt{\lambda n}} \mbox{ or } x_i = 0 \right\},\\
B_{2,\infty} &=& \left\{x \in \R^n: \norm{ x}_2 \leq 1, \norm{x}_{\infty} < \frac{1}{\sqrt{\lambda n}} \right\}.
\end{eqnarray*}
Given the above two sets, we can apply an $\epsilon$-net argument to each set separately. The advantage is that since vectors on $B_{2,0}$ only have a small number of non-zero entries, the size of the $\epsilon$-net on $B_{2,0}$ is small. This counteracts the fact that the measure concentration bound that we get for vectors in $B_{2,0}$ is rather weak since the vectors in this set have arbitrarily large entries (upper bounded by one). On the other hand, vectors in $B_{2,\infty}$ have many non-zero coefficients of bounded magnitude. As a result, the cardinality of the $\epsilon$-net on $B_{2,\infty}$ is large, but the measure concentration bound is much tighter. Combining the contribution of the sparse and the spread vectors results to a strong overall bound.

We conclude the section by noting that the above two sets are spanning the whole unit sphere $\Sphere^{n-1}$ in $\R^n$. Using the inequality $\left(\E (x + y)^q\right)^{1/q} \leq \left(\E x^q\right)^{1/q} + \left(\E y^q\right)^{1/q}$ we obtain
\begin{eqnarray}
\label{ineq::divide space of x and y into 4 subsets}
\label{eqn:term1}\left( \E \sup_{ x,  y \in \Sphere^{n-1}} \norm{\tensor H \times_1 x \times_2 y }^q_2 \right)^{1/q} &\leq& \left( \E \sup_{ x,  y \in B_{2,0}} \norm{\tensor H \times_1 x \times_2 y}^q_2 \right)^{1/q} \\
\label{eqn:term2}&+& \left( \E \sup_{ x,  y \in B_{2,\infty}} \norm{\tensor H \times_1 x \times_2 y}^q_2 \right)^{1/q} \\
\label{eqn:term3}&+& \left( \E \sup_{ x \in B_{2,0},  y \in B_{2,\infty}} \norm{\tensor H \times_1 x \times_2 y}^q_2 \right)^{1/q} \\
\label{eqn:term4}&+& \left( \E \sup_{ x \in B_{2,\infty},  y \in B_{2,0}} \norm{\tensor H \times_1 x \times_2 y}^q_2 \right)^{1/q}.
\end{eqnarray} 
\subsubsection{Controlling sparse vectors}\label{sxn:sparse}

\noindent We now prove the following lemma bounding the contribution of the sparse vectors (term~(\ref{eqn:term1})) in our $\epsilon$-net construction.
\begin{lem}\label{prop::Control sparse vectors - final bound for expected norm of B(x,y) over B_20}
Consider a $d$-mode tensor $\tensor A$ and let $\tensor H$ be the $d$-mode tensor after the Gaussian symmetrization argument as defined in Section~\ref{sxn:GaussianRandomTensor}. Let $\alpha$ and $\beta$ be
\begin{eqnarray}
\label{eqt::define alpha} \alpha^2 &=& \max\left\{\max_{i,j} \sum_{k=1}^n\tensor A_{ijk}^2, \max_{i,k} \sum_{j=1}^n\tensor A_{ijk}^2, \max_{j,k} \sum_{i=1}^n\tensor A_{ijk}^2\right\},\\
\label{eqt::define beta} \beta &=& \max_{i,j,k} \abs{\tensor A_{ijk}}.
\end{eqnarray}
For all $q \geq 1$,
\begin{equation}
\left( \E \sup_{ x,  y \in B_{2,0}} \norm{\tensor H \times_1 x \times_2 y}_2^q \right)^{1/q} \leq (3\sqrt{q})^{1/q} 2^{(d-1)} \left(\alpha +  \beta \sqrt{2 d \lambda n \ln \frac{5e}{\lambda}} + \beta \sqrt{q} \right).
\end{equation}
\end{lem}

The expectation bound has two components: the first one relates to the maximum tensor row or column energy and the second one relates to largest entry of the tensor. While the first component involving $\alpha$ is fixed, the size of the set $B_{2,0}$ affects the second component which involves $\beta$. Roughly speaking, the above expectation bound is of the order of $\alpha + \beta \sqrt{\lambda n \ln 1/\lambda}$. It is also clear that $\lambda$ control the size of the set $B_{2,0}$: smaller $\lambda$ is associated with a smaller set $B_{2,0}$. If the entries of the tensor are spread out, then $\alpha \approx \beta \sqrt{n}$ and we can set $\lambda$ to be a large constant and the expectation bound is optimal $O\left(\alpha\right)$. On the other hand, we can select a smaller value for $\lambda = c/n$ to get the bound $\alpha + \beta \sqrt{\ln n}$. We also emphasize that $\sup_{ x,  y \in B_{2,0}} \norm{\tensor H \times_1 x \times_2 y}_2$ is lower bounded by $\alpha$, which can be seen by setting $x$ and $y$ to be basis vectors. Therefore, the above expectation bound is tight.

\begin{proof}
Let $K = \lambda n$ and let $B_{2,0,K}$ be the $K$-dimensional set defined by
$$
B_{2,0,K} = \{ x \in \R^K: \norm{ x}_2 \leq 1  \}.
$$
Then, the set $B_{2,0}$ corresponding to vectors with at most $K$ non-zero entries can be expressed as a union of subsets of dimension $K$, i.e.,
$
B_{2,0} = \bigcup B_{2,0,K}.
$
A simple counting argument indicates that there are at most $\binom{n}{K} \leq \left(\frac{en}{K}\right)^K$ such subsets. We now apply the $\epsilon$-net technique to each of the subsets $B_{2,0,K}$ whose union is the set $B_{2,0}$. First, let us define $N_{B_{2,0,K}}$ to be the $1/2$-net of a subset $B_{2,0,K}$. Lemma 3.18 of~\cite{Ledoux_2001_ConcentrationMeasure_book} bounds the cardinality of $N_{B_{2,0,K}}$ by $5^K$. Applying Lemma~\ref{prop::dicretize of norm of a tensor A} with $\epsilon=1/2$ we get
$$
\sup_{x,  y \in B_{2,0,K}} \norm{\tensor H \times_1 x \times_2 y}_2 \leq 2^{d-1} \sup_{x, y \in N_{B_{2,0,K}}} \norm{\tensor H \times_1 x \times_2 y}_2.
$$
The right-hand side can be controlled by Lemma \ref{lem::bound norm of B times x times y in probability} which bounds the term $\tensor H \times_1 x \times_2 y$ for a specific pair of unit vectors $x$ and $y$. Noticing that
$$
\max_k \left(\sum_{i,j} \tensor A^2_{ijk} x^2_i y^2_j \right)^{1/2} \leq \max_{i,j,k} \abs{\tensor A_{ijk}} \left(\sum_{i,j} x^2_i y^2_j \right)^{1/2}\leq
\max_{i,j,k} \abs{\tensor A_{ijk}} = \beta,
$$
we apply Lemma \ref{lem::bound norm of B times x times y in probability} and take the union bound over all $x, y \in N_{B_{2,0,K}}$ to yield
$$
\Prob \left( \sup_{ x,  y \in B_{2,0,K}} \norm{\tensor H \times_1 x \times_2 y}_2  \geq 2^{d-1}\left(\alpha + t \sqrt{2} \beta \right) \right) \leq \left(5^K \right)^{d-1} e^{-t^2}.
$$
In the above $\alpha$ and $\beta$ are defined in eqns.~(\ref{eqt::define alpha}) and~(\ref{eqt::define beta}) respectively. We now explain the $\left(5^K \right)^{d-1}$ term in the failure probability. In general, the product $\tensor H \times_1 x \times_2 y \cdots$ should be evaluated on $d-1$ vectors $x,y,\ldots$. Recall that the $1/2$-net $N_{B_{2,0,K}}$ contains $5^K$ vectors and thus there is a total of $\left(5^K \right)^{d-1}$ possible vector combinations. A standard union bound now justifies the above formula. Finally, taking the union bound over all possible subsets $B_{2,0,K}$ that comprise the set $B_{2,0}$ and using $K = \lambda n$ yields
\begin{eqnarray}
\nonumber \Prob \left( \sup_{ x,  y \in B_{2,0}} \norm{\tensor H \times_1 x \times_2 y}_2  \geq 2^{d-1}\left(\alpha + t \sqrt{2} \beta \right) \right) &\leq &
\left(\left(\frac{en}{K}\right)^K\right)^{d-1}\left(5^K \right)^{d-1} e^{-t^2}\\
\nonumber &=&\left(\frac{5e}{\lambda}\right)^{\lambda n(d-1)}e^{-t^2}\\
\label{ineq::bound B(x,y) in the subspace B_20} &\leq&\left(\frac{5e}{\lambda}\right)^{\lambda nd}e^{-t^2}.
\end{eqnarray}
In the above, we again accounted for all $d-1$ modes of the tensor and also used $d-1 \leq d$. Using eqn.~(\ref{ineq::bound B(x,y) in the subspace B_20}) and applying Lemma \ref{lem::convert from probabilistic bound to expectation bound} (part (b)) with $a = 2^{d-1} \alpha$, $b = 2^{d-1}\beta \sqrt{2}$, and $h = d \lambda n \ln (5 e/\lambda)$ we get
\begin{equation*}
\label{ineq::bound E B(x,y) in the subspace B_20}
\E \sup_{ x,  y \in B_{2,0}} \norm{\tensor H \times_1 x \times_2 y}^q_2 \leq 3\sqrt{q} \left(2^{d-1}\left(\alpha +  \beta \sqrt{2 d \lambda n \ln (5 e/\lambda)} + \beta \sqrt{q} \right) \right)^q.
\end{equation*}
Raising both sides to $1/q$ completes the proof.
\end{proof}

\subsubsection{Controlling spread vectors}\label{subsect::Control spread vectors}

\noindent We now prove the following lemma bounding the contribution of the spread vectors (term~(\ref{eqn:term2})) in our $\epsilon$-net construction.
\begin{lem}\label{prop::Control spread vectors - final bound for expected norm of B(x,y) over B_20 infty}
Consider a $d$-mode tensor $\tensor A$ and let $\tensor H$ be the $d$-mode tensor after the Gaussian symmetrization argument as defined in Section~\ref{sxn:GaussianRandomTensor}. Let $\alpha$ be defined as in eqn.~(\ref{eqt::define alpha}). For all $q \geq 1$,
\begin{equation}\label{ineq::bound spread vectors - bound Expectation sup norm B(x,y) onto B_2 infty}
\left(\E \sup_{x, y \in B_{2, \infty}} \norm{\tensor H \times_1 x \times_2 y}^q_2\right)^{1/q} \leq (3\sqrt{q})^{1/q} \,\,\, 4^{d-1} \left( \log_2 \frac{1}{\lambda} \right)^{d-1} \alpha \left(1 + \sqrt{2d\ln \frac{2e}{\lambda} } + \sqrt{\frac{q}{\lambda n}}\right) ,
\end{equation}
assuming that $\lambda \leq 1/64$.
\end{lem}
It is worth noting that the particular choice of the upper bound for $\lambda$ is an artifact of the analysis and that we could choose bigger values for $\lambda$ by introducing a constant factor loss in the above inequality.

\begin{proof}
Our proof strategy is similar to the one used in Lemma~\ref{prop::Control sparse vectors - final bound for expected norm of B(x,y) over B_20}. However, in this case, the construction of the $\epsilon$-net for the set $B_{2,\infty}$ is considerably more involved. Recall the definition of $B_{2,\infty}$:
\begin{eqnarray*}
B_{2,\infty} &=& \left\{x \in \R^n: \norm{ x}_2 \leq 1, \norm{x}_{\infty} < \frac{1}{\sqrt{\lambda n}} \right\}.
\end{eqnarray*}
We now define the following sets of vectors $N_k$ with $k = 0,1,...,2 M-1$ with $M \triangleq \lceil 2 + \log_2 1/\sqrt{\lambda} \rceil$, assuming that $\lambda \leq 1$:
\begin{eqnarray*}
N_k &=& \{z \in B_{2,\infty}: \text{ for all } i \in [n],\ z_i = \pm \frac{1}{2^{k/2} \sqrt{\lambda n}} \text{ or } z_i = 0 \}.\\
\end{eqnarray*}
Our $\frac{1}{2}$-net for $B_{2,\infty}$ will be the set
\begin{eqnarray*}
N_{B_{2,\infty}} &=& \{z \in B_{2,\infty}: \text{ for all } i \in [n],\ z_i = \pm \frac{1}{2^{k/2}\sqrt{\lambda n}} \text{ with either } k=0,1,...,2M-1  \text{ or } z_i = 0 \},\\
\end{eqnarray*}
Our first lemma argues that $N_{B_{2,\infty}}$ is indeed a $\frac{1}{2}$-net for $B_{2,\infty}$.
\begin{lem}
Assuming $\lambda \leq 1$. For all $x \in B_{2,\infty}$ there exists a vector $z \in N_{B_{2,\infty}}$ such that
$$
\norm{x - z}_{\infty} \leq \frac{1}{2\sqrt{\lambda n}} \quad\quad \text{and} \quad\quad \norm{x-z}_2 \leq \frac{1}{2}.
$$
\end{lem}
\begin{proof}
Consider a vector $x \in B_{2,\infty}$ with coordinates $x_i$ for all $i \in [n]$. If $\frac{1}{2^{(k+1)/2}\sqrt{\lambda n}} \leq \abs{x_i} < \frac{1}{2^{k/2} \sqrt{\lambda n}}$ for some $k = 0,1,...,2M-1$, then we set $z_i = \sgn{x_i}\frac{1}{2^{(k+1)/2}\sqrt{\lambda n}}$. It is clear from this construction that
$$
|x_i - z_i| \leq \frac{1}{2^{k/2} \sqrt{\lambda n}} - \frac{1}{2^{(k+1)/2}\sqrt{\lambda n}} =\frac{ \sqrt{2}-1}{2^{(k+1)/2}\sqrt{\lambda n}} \leq  (\sqrt{2}-1) |x_i|.
$$
On the other hand, if $\abs{x_i} < \frac{1}{2^M\sqrt{\lambda n}}$ then we set $z_i = 0$. It is also clear that
$$
|x_i - z_i| <  \frac{1}{2^{\lceil 2 + \log_2 1/\sqrt{\lambda} \rceil}\sqrt{\lambda n}} \leq \frac{1}{2^{2 + \log_2 1/\sqrt{\lambda} }\sqrt{\lambda n}} = \frac{1}{4\sqrt{n}}.
$$
This choice of $z$ is clearly in $N_{B_{2,\infty}}$ and implies that for all $i \in [n]$,
$$
\abs{x_i - z_i} \leq \max \left\{ (\sqrt{2}-1) |x_i|,  \frac{1}{4\sqrt{n}} \right\} \leq \frac{1}{2\sqrt{\lambda n}}.
$$
In addition, $(x_i-z_i)^2 \leq \max \{(\sqrt{2}-1)^2 x_i^2, \frac{1}{16n} \} \leq (\sqrt{2}-1)^2 x_i^2 + \frac{1}{16n}$ implies that
$$
\norm{x - z}_2^2 \leq \sum_{i=1}^n \left((\sqrt{2}-1)^2 x_i^2 + \frac{1}{16n} \right) = \frac{1}{16} + (\sqrt{2}-1)^2 \norm{x}_2^2 < \frac{1}{4},
$$
which concludes the lemma.
\end{proof}

Given our definitions for $N_k$ and $N_{B_{2,\infty}}$, it immediately follows that any vector in $N_{B_{2,\infty}}$ can be expressed as a sum of $2M$ vectors, each in $N_k$ with $k = 0,1,...,2M-1$. Combining the above lemma with Lemma~\ref{prop::dicretize of norm of a tensor A}, we get
\begin{eqnarray*}
\nonumber \sup_{x, y \in B_{2,\infty}} \norm{\tensor H \times_1 x \times_2 y}_2 &\leq& 2^{d-1} \sup_{x,y \in N_{B_{2,\infty}}} \norm{\tensor H \times_1 x \times_2 y}_2 \\
&\leq& 2^{d-1} \sum_{k=0}^{2M-1} \,\, \sum_{k'=0}^{2M-1} \sup_{x\in N_k, y \in N_{k'}} \norm{\tensor H \times_1 x \times_2 y}_2 .
\end{eqnarray*}
We notice here that there are two summations associated with $k$ and $k'$. However, for general order-$d$ tensor, the total summations are $(d-1)$. We now raise both sides of the above inequality to the $q$-th power. In order to get a meaningful bound, we employ the following lemma, which is a direct consequence of the H\"older's inequality.
\begin{lem}
\label{lem::bound sum of power of q}
Let $a_i$, $i=1,...,n$ be nonnegative number. For any $q \geq 1$,
$$
\left(\sum_{i=1}^n a_i \right)^q \leq n^{q-1} \left(\sum_{i=1}^n a_i^q \right).
$$
\end{lem}
\noindent Applying Lemma \ref{lem::bound sum of power of q}, we get
\begin{equation} \label{eqn:pd102}
\sup_{x, y \in B_{2,\infty}} \norm{\tensor H \times_1 x \times_2 y}_2^q \leq 2^{q(d-1)} (2M)^{2(q-1)} \left(\sum_{k=0}^{2M-1} \sum_{k'=0}^{2M-1} \sup_{x\in N_k, y \in N_{k'}} \norm{\tensor H \times_1 x \times_2 y}_2^q	\right).
\end{equation}
It is important to note that in the general case of order-$d$ tensors we would have a total of $\left(2M\right)^{(d-1)(q-1)}$ terms involving $(d-1)$ summations (as opposed to $\left(2M\right)^{2(q-1)}$ in the case of order-3 tensors). Our final bound accounts for all these terms and we will return to this point later in this section. Our next lemma bounds the number of vectors in $N_k$.
\begin{lem}\label{lem:netsizes}
Given our definitions for $N_k$,
$\abs{N_k } \leq e^{2^k\lambda n \ln (2e/\lambda)}$.
\end{lem}
\begin{proof}
For all $z \in N_k$, the number of non-zero entries in $z$ is at most $2^k \lambda n$, since $\norm{z}_2 \leq 1$. Let $\gamma = 2^k \lambda n$ and notice that the number of non-zero entries in $z$ (the ``sparsity'' of $z$, denoted by $s$) can range from $1$ up to $\min(\gamma,n)$. For each value of the sparsity parameter $s$, there exist $2^s\binom{n}{s}$ choices for the non-zero coordinates ($\binom{n}{s}$ positions times $2^s$ sign choices). Thus, for $k$ such that $\gamma \leq n$, the cardinality of $N_k$ is bounded by
\begin{equation}
\begin{split}
\abs{N_k } \leq \sum_{s = 1}^{\gamma} \binom{n}{s} 2^s \leq \left( \frac{2 e n}{\gamma} \right)^\gamma = \left( \frac{2 e n}{2^k \lambda n} \right)^\gamma \leq \left( \frac{2 e}{\lambda} \right)^\gamma
% e^{\gamma \ln (2en/\gamma)} = e^{2^k\lambda n \ln (2e/ 2^k \lambda)} \leq e^{2^k\lambda n \ln (2e/\lambda)}.
\end{split}
\end{equation}
Similarly, for $k$ such that $\gamma \geq n$, $\abs{N_k } \leq \sum_{s = 1}^{n} \binom{n}{s} 2^s = 3^n $ which is also less than $ \left( \frac{2 e}{\lambda} \right)^\gamma$ for $\lambda \leq 1$. In both cases, we have $\abs{N_k } \leq e^{\gamma \ln (2e/\lambda)} = e^{2^k\lambda n \ln (2e/\lambda)}$, as claimed.
\end{proof}

\noindent We now proceed to estimate the quantity $\norm{\tensor H \times_1 x \times_2 y}_2$ over all vector combinations that appear in eqn.~(\ref{eqn:pd102}).
\begin{lem}\label{lem:point1}
Using our notation, for any fixed $k$ and $k'$ in $(0,1,...,2M-1)$
\begin{equation}
\label{inq::exect norm of the sup}
\E \sup_{(x,y) \in (N_k, N_{k'})} \norm{\tensor H \times_1 x \times_2 y}_2^q \leq 3\sqrt{q}\left( \alpha + \alpha\sqrt{2d\ln(2e/\lambda)} + \alpha\sqrt{\frac{q}{\lambda n}}  \right)^q.
\end{equation}
\end{lem}
\begin{proof}
\noindent Without loss of generality, we assume $k \geq k'$. We first establish the probability bound via Lemma~\ref{lem::bound norm of B times x times y in probability} and then apply Lemma \ref{lem::convert from probabilistic bound to expectation bound} to obtain the expectation estimate. We have,
\begin{eqnarray*}
\max_l \left( \sum_{i,j}\tensor A^2_{ijl}x_i^2 y_j^2 \right) 
&=& \max_l \left( \sum_i x_i^2 \sum_j y^2_j \tensor A^2_{ijl} \right) \\
&\leq& \max_l \frac{1}{2^k\lambda n} \left(\sum_j y^2_j \sum_i \tensor A^2_{ijl} \right) \\
&\leq&  \frac{1}{2^k\lambda n} \max_{j,l}\sum_{i}\tensor A_{ijl}^2 .
\end{eqnarray*}
In the above we used the fact that $\norm{y}_2 \leq 1$ and $\norm{x}_{\infty} = \frac{1}{2^{k/2}\sqrt{\lambda n}}$. Applying Lemma~\ref{lem::bound norm of B times x times y in probability}, we get (recall the definition of $\alpha$ from eqn.~(\ref{eqt::define alpha})):
\begin{equation}
\label{ineq::Control spread vectors - probabilistic bound norm B(x,y)}
\Prob \left( \norm{\tensor H \times_1 x \times_2 y}_2 \geq  \alpha  +  t \sqrt{2} \frac{1}{2^{k/2}\sqrt{\lambda n}} \alpha \right) \leq e^{-t^2}.
\end{equation}
Taking the union bound over all possible combinations of vectors $x \in N_k$ and $y \in N_{k'}$ and using Lemma~\ref{lem:netsizes} and the fact that $|N_k| \leq e^{2^k \lambda n \ln(2e/\lambda)} $, we get
$$
\Prob \left( \sup_{x,y \in N_k} \norm{\tensor H \times_1 x \times_2 y}_2 \geq  \alpha  +  t \sqrt{2} \frac{1}{2^{k/2}\sqrt{\lambda n}} \alpha \right) \leq e^{-t^2 + (d-1) 2^k \lambda n \ln(2e/ \lambda)},
$$
where the $(d- 1)$ factor appears in the exponential because of a union bound over all $(d-1)$ vectors that could appear in the product $\tensor H \times_1 x \times_2 y \times_3\cdots$.

\noindent To prove the expectation bound, we apply Lemma \ref{lem::convert from probabilistic bound to expectation bound} with $a = \alpha$, $b = \frac{\sqrt{2}}{2^{k/2}\sqrt{\lambda n}} \alpha$ and $h = (d-1) 2^k \lambda n \ln(2e/ \lambda)$ to get
\begin{equation}
\begin{split}
\nonumber
\E \sup_{(x,y) \in (N_k, N_{k'})} \norm{\tensor H \times_1 x \times_2 y}_2^q &\leq 3\sqrt{q}\left(a + b\sqrt{h} + b\sqrt{q/2}\right)^q \\
&= 3\sqrt{q}\left( \alpha + \alpha\sqrt{2(d-1)\ln(2e/\lambda)} + \alpha\sqrt{\frac{q}{2^k \lambda n}}  \right)^q.
\end{split}
\end{equation}
Proving the lemma is now trivial using $d-1 \leq d$ and $2^k \geq 1$ for all $k \geq 0$.
\end{proof}

\noindent Using the bounds of Lemma~\ref{lem:point1} and combining with eqn.~(\ref{eqn:pd102}), we get
\begin{equation}
\begin{split}
\E \sup_{x, y \in B_{2, \infty}} \norm{\tensor H \times_1 x \times_2 y}^q_2 &\leq 2^{q(d-1)} (2M)^{2(q-1)} \\
&\times\left(\sum_{k=0}^{2M-1} \sum_{k'=0}^{2M-1} 3\sqrt{q} \left( \alpha + \alpha\sqrt{2d\ln(2e/\lambda)} + \alpha\sqrt{\frac{q}{ \lambda n}}  \right)^q \right) \\
&= 3 \times 2^{q(d-1)} (2M)^{2q} \sqrt{q} \left( \alpha + \alpha\sqrt{2d\ln(2e/\lambda)} + \alpha\sqrt{\frac{q}{ \lambda n}}  \right)^q.
\end{split}
\end{equation}
We note that in the last equation, the number two that appears in the exponent of the term $2M$ accounts for the two summations associated with $x \in N_k$ and $y \in N_{k'}$. In general, for order-$d$ tensors, there are at most $(d-1)$ such summations. Therefore, after some rearranging of terms,
$$
\E \sup_{x, y \in B_{2, \infty}} \norm{\tensor H \times_1 x \times_2 y}^q_2 \leq 3\sqrt{q} \left( 2^{d-1} (2M)^{d-1} \left( \alpha + \alpha\sqrt{2d\ln(2e/\lambda)} + \alpha\sqrt{\frac{q}{ \lambda n}}  \right) \right)^q.
$$
To conclude the proof of Lemma~\ref{prop::Control spread vectors - final bound for expected norm of B(x,y) over B_20 infty} we use our assumption on $\lambda$ and the following inequality:
$$M = \lceil 2+ \log_2 \frac{1}{\sqrt{\lambda}} \rceil \leq 3 + \log_2 \frac{1}{\sqrt{\lambda}} \leq 2 \log_2 \frac{1}{\sqrt{\lambda}} = \log_2 1/\lambda.$$
\end{proof}
\subsubsection{Controlling combinations of sparse and spread vectors }
\label{subsect::Control combinations of sparse and spread vectors}

\noindent We now prove the following lemma bounding the contribution of combinations of sparse and spread vectors (terms~(\ref{eqn:term3}) and (\ref{eqn:term4})) in our $\epsilon$-net construction.

\begin{lem}\label{prop::Control sparse and spread vectors - final bound for expected norm of B(x,y) over B_20 B_2 infty}
Consider a $d$-mode tensor $\tensor A$ and let $\tensor H$ be the $d$-mode tensor after the Gaussian symmetrization argument as defined in Section~\ref{sxn:GaussianRandomTensor}. Let $\alpha$ be defined as in eqn.~(\ref{eqt::define alpha}). For all $q \geq 1$,
\begin{equation}
\left(\E \sup_{x \in B_{2, 0}, y \in B_{2, \infty}} \norm{\tensor H \times_1 x \times_2 y}^q_2\right)^{1/q} \leq (3\sqrt{q})^{1/q} \,\,\, 4^{d-1} \left( \log_2 \frac{1}{\lambda} \right)^{d-2} \alpha \left(1 + \sqrt{2d\ln \frac{5e}{\lambda}} + \sqrt{\frac{q}{\lambda n}}\right) ,
\end{equation}
assuming that $\lambda \leq 1/64$.
\end{lem}
\noindent It is worth noting that the particular choice of the upper bound for $\lambda n$ is an artifact of the analysis and that we could choose bigger values for $\lambda n$ by introducing a constant factor loss in the above inequality.

\begin{proof}
Let $x \in B_{2,0}$ and $ y \in B_{2,\infty}$. In Sections~\ref{sxn:sparse} and~\ref{subsect::Control spread vectors} we defined $N_{B_{2,0}}$ (a $1/2$-net for $B_{2,0}$) and $N_{B_{2,\infty}}$ (a $1/2$-net for $B_{2,\infty}$). Recall that for $K = \lambda n$, $B_{2,0}$ was the union of $\binom{n}{K}$ $K$-dimensional subsets $B_{2,0,K}$. Consequently, the $1/2$-net $N_{B_{2,0}}$ is the union of the $1/2$-nets $N_{B_{2,0,K}}$ (each $N_{B_{2,0,K}}$ is the $1/2$-net of $B_{2,0,K}$). Recall from Section~\ref{sxn:sparse} that the cardinality of $N_{B_{2,0}}$ is bounded by
\begin{equation}\label{eqn:nn1}
\abs{N_{B_{2,0}}} = \binom{n}{K} \abs{N_{B_{2,0,K}}} \leq \left(\frac{en}{K}\right)^K 5^K = \left(\frac{5e}{\lambda}\right)^{\lambda n}.
\end{equation}
We apply Lemma~\ref{prop::dicretize of norm of a tensor A} to get
\begin{equation}\label{ineq::Control sparse and spread vectors - bound sup norm B(x,y) in space B_20 B_2 intfy}
\sup_{ x \in B_{2,0} ,y \in B_{2,\infty} } \norm{\tensor H \times_1 x \times_2 y}_2 \leq 2^{d-1} \sup_{ x \in N_{B_{2,0}} , y \in N_{B_{2,\infty}}}  \norm{\tensor H\times_1 x\times_2 y}_2.
\end{equation}
It is now important to note that for a general $d$-mode tensor $\tensor H$ the above product $\tensor H \times_1 x \times_2 y \times_3 \cdots$ would be computed over $d-1$ vectors, with at least one those vectors (w.l.o.g. $x$) in $N_{B_{2,\infty}}$ and at least one of those vectors (w.l.o.g. $y$) in $N_{B_{2,0}}$. Each of the remaining $(d-3)$ vectors could belong either to $N_{B_{2,0}}$ or to $N_{B_{2,\infty}}$. In order to proceed with our analysis, we will need to further express the vectors belonging to $N_{B_{2,\infty}}$ as a sum of $2M$ vectors belonging to $N_k$ with $k=0,1,...,2M-1$ and $M = \lceil 2+\log_2 1/\sqrt{\lambda} \rceil$, respectively. (The reader might want to recall our definition for $N_k$ from Section~\ref{subsect::Control spread vectors}). We note that the cardinality upper bound of the set $N_{B_{2,\infty}}$ is considerably larger than that of the set $N_{B_{2,0}}$. This can be easily seen by comparing the upper bound of $|N_{k}|$ in Lemma \ref{lem:netsizes} with that of $|N_{B_{2,0}}|$ in (\ref{eqn:nn1}). Therefore, we only need to consider the worse case scenario, in which all $(d-2)$ vectors in the product $\tensor H \times_1 x \times_2 y \times_3 \cdots$ belong to $N_{B_{2,\infty}}$. The bound for other cases will be smaller than the bound under consideration. The product can be expressed as a sum of (at most) $(2M)^{d-2}$ terms as follows:
$$\tensor H\times_1 x\times_2 y \cdots \times_d z  = \sum_{k=1}^{2M-1} \cdots \sum_{k'=1}^{2M-1} \tensor H \times_1 x \times_2 y_k \cdots \times_d z_{k'}
$$
where $x \in N_{B_{2,0}}$, $y, z \in N_{B_{2,\infty}}$, $y_k \in N_k$, and $z_{k'} \in N_{k'}$. Therefore, applying Lemma \ref{lem::bound sum of power of q} and taking the expection, we get
\begin{equation}
\begin{split}
\label{ineq::Control sparse and spread vectors - bound sup norm B(x,y) in space B_20 B_2 intfy - eqn2}
\E &\sup_{ x \in N_{B_{2,0}}, y \in N_{B_{2,\infty}}}  \norm{\tensor H\times_1 x\times_2 y \cdots}_2^q \\
&\leq (2M)^{(d-2)(q-1)} \left( \sum_{k=1}^{2M-1} \cdots \sum_{k'=1}^{2M-1} \E \sup_{x \in N_{B_{2,0}}, y \in N_k,...,z \in N_{k'}}  \norm{\tensor H\times_1 x \times_2 y \cdots \times_d z}_2^q \right).
\end{split}
\end{equation}

\noindent We now need a bound, in expectation, for the $q$-th power of the $\ell_2$ norm for each of the $(2M)^{d-2}$ terms. Fortunately, this bound has essentially already been derived in Section~\ref{subsect::Control spread vectors}. We start by noting that the bound of eqn.~(\ref{ineq::Control spread vectors - probabilistic bound norm B(x,y)}) holds when \textit{at least one} of the vectors in the product $\tensor H\times_1 x\times_2 y \cdots $ belongs to $N_k$. Thus,
\begin{equation}\label{ineq::Control sparse and spread vectors - probabilistic bound norm B(x,y)}
\Prob \left( \norm{\tensor H \times_1 x \times_2 y \cdots \times_d z}_2 \geq  \alpha  +  t \sqrt{2} \frac{1}{2^{\max\{k,...,k' \}/2}\sqrt{\lambda n}}\alpha \right) \leq e^{-t^2}
\end{equation}
holds for any $x \in N_{B_{2,0}}$, $y \in N_k$,..., and $z \in N_{k'}$. We apply a union bound by noting that from Lemma~\ref{lem:netsizes} the cardinalities of $N_k$ are upper bounded by $e^{2^k \lambda n \ln (e/2^{k-1} \lambda)} \leq e^{2^k \lambda n \ln (2e/\lambda)}$. Combining with eqn.~(\ref{eqn:nn1}) we get that the total number of possible vectors over which the \textit{sup} of eqn.~(\ref{ineq::Control sparse and spread vectors - probabilistic bound norm B(x,y)}) is computed does not exceed
\begin{equation}
\begin{split}
\nonumber
\left(\frac{5e}{\lambda}\right)^{\lambda n} \underbrace{ e^{2^k \lambda n  \ln(2e/\lambda)} \cdots e^{2^{k'} \lambda n \ln(2e/\lambda)} }_{(d-2) \text{ terms}} \leq e^{(d-1) 2^{\max\{k,...,k' \}} \lambda n  \ln(5e/\lambda) }.
\end{split}
\end{equation}
We can now use a standard union bound over all $x \in N_{B_{2,0}}$, $y \in N_k$,..., and $z \in N_{k'}$ to get
\begin{eqnarray*}
\Prob \left( \sup\norm{\tensor H\times_1 x_i\times_2 y \cdots}_2 \geq  \alpha  +  t \sqrt{2} \frac{1}{2^{\max\{k,...,k' \}/2} \sqrt{\lambda n}} \alpha \right) &\leq& e^{-t^2 + (d-1) 2^{\max\{k,...,k' \}} \lambda n  \ln(5e/\lambda) }.
\end{eqnarray*}
We are now ready to apply Lemma \ref{lem::convert from probabilistic bound to expectation bound} with $h = (d-1) 2^{\max\{k,...,k' \}} \lambda n  \ln(5e/\lambda) $, $a = \alpha$ and $b = \frac{\sqrt{2}}{2^{\max\{k,...,k' \}/2}\sqrt{\lambda n}} \alpha$ to get
$$
\E \sup_{ x \in N_{B_{2,0}},  y \in N_k, \cdots}  \norm{\tensor H\times_1 x\times_2 y\cdots}_2^q \leq 3\sqrt{q}\left(a + b\sqrt{h} + b\sqrt{q/2}\right)^q.
$$
Combining with eqns.~(\ref{ineq::Control sparse and spread vectors - bound sup norm B(x,y) in space B_20 B_2 intfy}) and (\ref{ineq::Control sparse and spread vectors - bound sup norm B(x,y) in space B_20 B_2 intfy - eqn2}) we get
$$
\E \sup_{ x \in B_{2,\infty},  y \in B_{2,0}} \norm{\tensor H\times_1 x\times_2 y\cdots}_2^q \leq 3\sqrt{q}\left(2^{d-1} (2M)^{d-2}\left(a + b\sqrt{h} + b\sqrt{q/2}\right)\right)^q.
$$
The proof follows by substituting the values of $a$, $b$, and $h$ in the above equation together with the fact that $M = \lceil 2+ \log_2 \frac{1}{\sqrt{\lambda}} \rceil \leq 3 + \log_2 \frac{1}{\sqrt{\lambda}} \leq 2 \log_2 \frac{1}{\sqrt{\lambda}} = \log_2 1/\lambda$.
\end{proof} 
\subsubsection{Concluding the proof of Theorem~\ref{thm::bound tensor (E norm(B) ^q)^(1/q)}}

Given the results of the preceding sections we can now conclude the proof of Theorem~\ref{thm::bound tensor (E norm(B) ^q)^(1/q)}. We combine Lemmas~\ref{prop::Control sparse vectors - final bound for expected norm of B(x,y) over B_20},~\ref{prop::Control spread vectors - final bound for expected norm of B(x,y) over B_20 infty}, and~\ref{prop::Control sparse and spread vectors - final bound for expected norm of B(x,y) over B_20 B_2 infty} in order to bound terms~(\ref{eqn:term1}), (\ref{eqn:term2}), (\ref{eqn:term3}), and (\ref{eqn:term4}). First,
\begin{equation*}
\begin{split}
(\E \norm{\tensor H}_2^q )^{1/q} \leq & (3\sqrt{q})^{1/q}  \,\,\, 2^{d-1} \left(\alpha +  \beta \sqrt{2 d \lambda n \ln (5 e/\lambda)} + \beta \sqrt{q} \right) \\
&+ (3\sqrt{q})^{1/q} \,\,\,  4^{d-1} \left( \log_2 1/\lambda \right)^{d-1} \left(\alpha + \alpha\sqrt{2d\ln \left(5e/\lambda\right)} + \alpha \sqrt{\frac{q}{\lambda n}}\right) \\
&+ \left(2^{d-1}-2\right) \times  (3\sqrt{q})^{1/q} \,\,\, 4^{d-1} \left( \log_2 1/\lambda \right)^{d-2} \left(\alpha + \alpha\sqrt{2d\ln \left(5e/\lambda\right)} + \alpha \sqrt{\frac{q}{\lambda n}}\right).
\end{split}
\end{equation*}
In the above bound we leveraged the observation that the right-hand side of the bound in Lemma~\ref{prop::Control sparse and spread vectors - final bound for expected norm of B(x,y) over B_20 B_2 infty} is also an upper bound for the right-hand side of the bound in Lemma~\ref{prop::Control spread vectors - final bound for expected norm of B(x,y) over B_20 infty} for all $\lambda \leq 1$. It is also crucial to note that the constant $2^{d-1}-2$ that appears in the second term of the above inequality emerges since for general order-$d$ tensors we would have to account for a total of $2^{d-1}$ terms in the last inequality of Section~\ref{sxn:ECintro}. Clearly, for order-3 tensors, this inequality has a total of four terms. Simplifying the right-hand side via the assumption $q \leq 2d \lambda n \ln (5e/\lambda)$ and the fact that $q^{1/q}$ is bounded by $e$, we obtain
\begin{equation}
\label{ineq::Control sparse and spread vectors - Final expected norm bound for B}
(\E \norm{\tensor H}_2^q )^{1/q} \leq c_1 8^{d-1} \left( \alpha [ \log_2 1/\lambda]^{d-1} + \beta \sqrt{\lambda n} \right) \sqrt{2d \ln(5e/\lambda)} ,
\end{equation}
where $c_1$ is a small constant. We now remind the reader that the entries $\tensor H_{ijk}$ of the tensor $\tensor H$ are equal to $g_{ijk} \tensor A_{ijk}$, where the $g_{ijk}$'s are standard Gaussian random variables. Thus,
$$
\E \norm{\tensor H}_2^q = \E_{g}\norm{\sum_{i,j,k} g_{ijk} \tensor A_{ijk} \cdot  e_i \otimes e_j \otimes e_k}_2^q.
$$
\noindent Substituting eqn.~(\ref{ineq::Control sparse and spread vectors - Final expected norm bound for B}) to eqn.~(\ref{ineq::bound expectation of (pi_omega - pI) (T) - step 4}) yields
\begin{equation}
\label{ineq::bound E || B || - last step}
\begin{split}
\left(\E_{\tensor A} \norm{\tensor A - \widehat{\tensor A}}_2^q\right)^{1/q} &\leq \sqrt{2 \pi}\left( \E_{\tensor A} \left[ c_1 8^{d-1} \left( \alpha [ \log_2 1/\lambda ]^{d-1} + \beta \sqrt{\lambda n} \right) \right]^q  \right)^{1/q} \\
&= c_2 8^d \sqrt{2d \ln \left( \frac{5e}{\lambda} \right)}  \left[ \E_{\tensor A} \left( \alpha [ \log_2 1/\lambda ]^{d-1} + \beta \sqrt{\lambda n} \right)^q \right]^{1/q} \\
&\leq c_3 8^d \sqrt{2d \ln \left( \frac{5e}{\lambda} \right)}  \left( [ \log_2 1/\lambda ]^{d-1} \left( \E_{\tensor A} \alpha^q \right)^{1/q}  + \sqrt{\lambda n} \left( \E_{\tensor A} \beta^q \right)^{1/q} \right),
\end{split} 
\end{equation}
where the last inequality follows from Lemma \ref{lem::bound sum of power of q}. Finally, we rewrite the $\E_{\tensor A} \alpha^q$ as
\begin{eqnarray*}
\E_{\tensor A} \alpha^q &=&  \E_{\tensor A} \max\left\{\max_{i,j} \left(\sum_{k=1}^n\tensor A_{ijk}^2\right)^{q/2}, \max_{i,k} \left(\sum_{j=1}^n\tensor A_{ijk}^2\right)^{q/2}, \max_{j,k} \left(\sum_{i=1}^n\tensor A_{ijk}^2\right)^{q/2}\right\}\\
&\leq& \E_{\tensor A} \max_{i,j} \left(\sum_{k=1}^n\tensor A_{ijk}^2\right)^{q/2} + \E_{\tensor A} \max_{i,k} \left(\sum_{j=1}^n\tensor A_{ijk}^2\right)^{q/2} + \E_{\tensor A} \max_{j,k} \left(\sum_{i=1}^n\tensor A_{ijk}^2\right)^{q/2}.
\end{eqnarray*}
More generally, for any order-$d$ tensor, we get
$$
\E_{\tensor A} \alpha^q \leq \sum_{j=1}^d \E_{\tensor A} \max_{i_1,\ldots,i_{j-1},i_{j+1},\ldots,i_d} \left(\sum_{i_j=1}^n \tensor A_{i_1 \ldots i_{j-1} i_j i_{j+1} ... i_d}^2\right)^{q/2}.
$$
Combining the above inequality and eqn.~(\ref{ineq::bound E || B || - last step}) concludes the proof of Theorem \ref{thm::bound tensor (E norm(B) ^q)^(1/q)}.

\section{Proving Theorem \ref{thm::main theorem of tensor sparsification}}

The main idea underlying our proof is the application of a divide-and-conquer-type strategy in order to decompose the tensor $\tensor A - \widetilde{\tensor A}$ as a sum of tensors whose entries are bounded. Then, we will apply Theorem \ref{thm::bound tensor (E norm(B) ^q)^(1/q)} and Corollary \ref{cor::bound tensor (E norm(B) ^q)^(1/q)} to estimate the spectral norm of each tensor in the summand independently.

To formally present our analysis, let $\tensor A^{[1]} \in \R^{n \times ... \times n }$  be a tensor containing all entries $\tensor A_{i_1...i_d}$ of $\tensor A$ that satisfy  $\tensor A_{i_1...i_d}^2 \geq  2^{-1} \frac{\norm{\tensor A}^2_F}{s}$; the remaining entries of $\tensor A^{[1]}$ are set to zero. Similarly, we let $\tensor A^{[k]} \in \R^{n \times ... \times n}$ (for all $k >1$) be tensors that contain all entries $\tensor A_{i_1...i_d}$ of $\tensor A$ that satisfy $\tensor A_{i_1...i_d}^2 \in \left[ 2^{-k} \frac{\norm{\tensor A}^2_F}{s}, 2^{-k+1} \frac{\norm{\tensor A}^2_F}{s} \right)$; the remaining entries of $\tensor A^{[k]}$ are set to zero. Finally, the tensors $\widetilde{\tensor A}^{[k]}$ (for all $k=1,2,\ldots$) contain the (rescaled) entries of the corresponding tensor $\tensor A^{[k]}$ that were selected after applying the sparsification procedure of Algorithm~1 to $\tensor A$. Given these definitions,
\begin{equation*}
\tensor A =  \sum_{k=1}^{\infty} \tensor A^{[k]} \qquad\text{and}\qquad \widetilde{\tensor A}=\sum_{k=1}^{\infty}\widetilde{\tensor A}^{[k]}.
\end{equation*}
Let $\ell \triangleq \left\lfloor \log_2 \left(n^{d/2} / \ln^d n\right) \right\rfloor$. Then,
\begin{eqnarray*}
\norm{\tensor A - \widetilde{\tensor A}}_2 &=& \norm{\sum_{k=1}^{\infty} \left(\tensor A^{[k]} - \widetilde{\tensor A}^{[k]}\right)}_2\\
&\leq& \norm{\tensor A^{[1]} - \widetilde{\tensor A}^{[1]}}_2 + \sum_{k=2}^{\ell} \norm{\tensor A^{[k]} - \widetilde{\tensor A}^{[k]}}_2
+ \norm{ \sum_{k=\ell+1}^{\infty} \left( \tensor A^{[k]} - \widetilde{\tensor A}^{[k]} \right) }_2.
\end{eqnarray*}
Using the inequality $(\E (x + y)^q)^{1/q} \leq (\E x^q)^{1/q} + (\E y^q)^{1/q}$, we conclude that
\begin{eqnarray}
\label{eqn:step1}\left(\E\norm{\tensor A - \widetilde{\tensor A}}_2^q\right)^{1/q} &\leq& \left(\E \norm{\tensor A^{[1]} - \widetilde{\tensor A}^{[1]}}_2^q\right)^{1/q}\\
\label{eqn:step2}&+& \sum_{k=2}^{\ell} \left(\E\norm{\tensor A^{[k]} - \widetilde{\tensor A}^{[k]}}_2^q\right)^{1/q}\\
\label{eqn:step3}&+& \left(\E \norm{\sum_{k=\ell+1}^{\infty}\left(\tensor A^{[k]} - \widetilde{\tensor A}^{[k]}\right)}^q_2 \right)^{1/q}.
\end{eqnarray}
The remainder of the section will focus on the derivation of bounds for terms~(\ref{eqn:step1}),~(\ref{eqn:step2}), and~(\ref{eqn:step3}) of the above equation.

\subsection{Term~(\ref{eqn:step1}): Bounding the spectral norm of $\tensor A^{[1]} - \widetilde{\tensor A}^{[1]}$}\label{sxn:tterm1}

The main result of this section is summarized in the following lemma.
\begin{lem}
\label{lem::final bound of E norm B^1}
Let $q \leq \frac{5n}{8}$. Then,
$$
\left(\E \norm{\tensor A^{[1]} - \widetilde{\tensor A}^{[1]}}^q \right)^{1/q} \leq  c_1 48^d d^{1/q+1/2} \sqrt{\frac{n \norm{\tensor A}_F^2}{s}},
$$
where $c_1$ is a small numerical constant.
\end{lem}
\begin{proof}
For notational convenience, let $\tensor B = \tensor A^{[1]} - \widetilde{\tensor A}^{[1]}$ and let $\tensor B_{i_1...i_d}$ denote the entries of $\tensor B$. Recall that $\tensor A^{[1]}$ only contains entries of $\tensor A$ whose squares are greater than or equal to $2^{-1} \frac{\norm{\tensor A}^2_F}{s}$. Also, recall that $\widetilde{\tensor A}^{[1]}$ only contains the (rescaled) entries of $\tensor A^{[1]}$ that were selected after applying the sparsification procedure of Algorithm~1 to $\tensor A$. Using these definitions, $\tensor B_{i_1...i_d}$ is equal to:
\begin{equation*}
\tensor B_{i_1...i_d} =
\begin{cases}
0 & \text{,if     } \tensor A_{i_1...i_d}^2 < 2^{-1}\frac{\norm{\tensor A}_F^2}{s} \\
0 & \text{,if     } \tensor A_{i_1...i_d}^2 \geq \frac{\norm{\tensor A}_F^2}{s} \text{\ \ \   (since $p_{i_1...i_d}=1$ )} \\
\left(1-p_{i_1...i_d}^{-1}\right) \tensor A_{i_1...i_d} & \text{,with probability     } p_{i_1...i_d} = \frac{s \tensor A^2_{i_1...i_d}}{\norm{\tensor A}^2_F} < 1\\
\tensor A_{i_1...i_d} & \text{,with probability   } 1-p_{i_1...i_d}
\end{cases}
\end{equation*}
It is easily seen from the formula of $\tensor B_{i_1...i_d}$ that $\tensor B_{i_1...i_d}^2 \leq \frac{\tensor A^2_{i_1...i_d}}{p^2_{i_1...i_d}} \leq \frac{\norm{\tensor A}_F^4}{s^2 \tensor A^2_{i_1...i_d}} \leq \frac{ \norm{\tensor A}_F^2}{s} $, which leads to
$$
\left( \E \tensor B_{i_1...i_d}^q \right)^{1/q} \leq \sqrt{\frac{\norm{\tensor A}^2_F}{s}}.
$$
\noindent In addition, we have for any $j$, $\max_{i_1,...,i_{j-1},i_{j+1},...,i_d} \sum_{i_j=1}^n \tensor B^2_{i_1...i_{j-1} i_{j+1} ...i_d} \leq \frac{n \norm{\tensor B}_F^2}{s}$, which leads to
\begin{equation}
\begin{split}
\nonumber
\left( \sum_{j=1}^d \E \max_{i_1,\ldots,i_{j-1},i_{j+1},\ldots,i_d} \left(\sum_{i_j=1}^n \tensor B_{i_1 \ldots i_{j-1} i_j i_{j+1} ... i_d}^2\right)^{q/2}\right)^{1/q} &\leq  \left( \sum_{j=1}^d \E \max_{i_1,\ldots,i_{j-1},i_{j+1},\ldots,i_d} \left(\sum_{i_j=1}^n \frac{\norm{\tensor A}_F^2}{s} \right)^{q/2}\right)^{1/q} \\
&\leq \left( d \left(n \frac{\norm{\tensor A}_F^2}{s} \right)^{q/2}\right)^{1/q} = d^{1/q} \sqrt{\frac{n \norm{\tensor A}_F^2 }{s}} .
\end{split}
\end{equation}

\noindent We will estimate the quantity $\left(\E \norm{\tensor B}_2^q \right)^{1/q}$ via Corollary \ref{cor::bound tensor (E norm(B) ^q)^(1/q)} as follows:
\begin{equation}\label{eqn:kequal1bound}
\begin{split}
\left( \E \norm{\tensor B}^q \right)^{1/q} \leq c 8^d \sqrt{2d \ln \left(\frac{5e}{\lambda} \right)} \left( \left[ \log_2 \frac{1}{\lambda}  \right]^{d-1} d^{1/q} \sqrt{\frac{n \norm{\tensor A}_F^2}{s}} + \sqrt{\lambda n} \sqrt{\frac{\norm{\tensor A}_F^2}{s}}\right).
\end{split}
\end{equation}

\noindent The proof follows by setting $\lambda = \frac{1}{64}$.
\end{proof}

\subsection{Term~(\ref{eqn:step2}): Bounding the spectral norm of $\tensor A^{[k]}-\widetilde{\tensor A}^{[k]}$ for small $k$}

We now focus on estimating the spectral norm of the tensors $\tensor A^{[k]} - \widetilde{\tensor A}^{[k]}$ for $2 \leq k \leq \left\lfloor \log_2 \left(n^{d/2} /\ln^{d/2} n\right) \right\rfloor$. The following lemma summarizes the main result of this section.
\begin{lem}
\label{lem::final bound of E norm B^k}
Assume that $q \leq 2d \lambda_k n \ln 5e/\lambda_k$; for all $2 \leq k \leq \left\lfloor \log_2 \left(n^{d/2} /\ln^{d/2} n\right) \right\rfloor$ and $\lambda_k \leq 1/64$,
$$
\left(\E \norm{\tensor A^{[k]}-\widetilde{\tensor A}^{[k]}}_2^q \right)^{1/q} \leq c_2 8^d \sqrt{2d \ln \left(\frac{5e}{\lambda_k} \right)} \left( \left[ \log_2 \frac{1}{\lambda_k} \right]^{d-1} (2d)^{\frac{1}{2q}} \sqrt{5n + (d\ln n+q)2^{k+1}} + \sqrt{\lambda_k 2^k n}  \right) \sqrt{\frac{\norm{\tensor A}^2_F}{s}},
$$
where $c_2$ is a small numerical constant.
\end{lem}
\begin{proof}
For notational convenience, we let $\widetilde{\tensor A}_{i_1...i_d}$ denote the entries of the tensor $\widetilde{\tensor A}^{[k]}$. Then,
\begin{equation}\label{eqn:define_tildeAij}
\widetilde{\tensor A}_{i_1 ... i_d} = \frac{\delta_{i_1...i_d} \tensor A_{i_1...i_d}}{p_{i_1...i_d}},
\end{equation}
for those entries $\tensor A_{i_1...i_d}$ of $\tensor A$ satisfying $\tensor A_{i_1...i_d}^2 \in \left[ \frac{\norm{\tensor A}^2_F}{2^k s}, \frac{\norm{\tensor A}^2_F}{2^{k-1} s}\right)$. All the entries of $\widetilde{\tensor A}^{[k]}$ that correspond to entries of $\tensor A$ outside this interval are set to zero. The indicator function $\delta_{i_1...i_d}$ is defined as
\begin{equation*}
\delta_{i_1...i_d} =
\begin{cases}
1 & \text{with probability     } p_{i_1...i_d} = \frac{s \tensor A^2_{i_1...i_d}}{\norm{\tensor A}^2_F} \leq 1\\
0 & \text{with probability   } 1-p_{i_1...i_d}
\end{cases}
\end{equation*}
Notice that $p_{i_1...i_d}$ is always in the interval $\left[2^{-k},2^{-(k-1)}\right)$ from the constraint on the size of $\tensor A_{i_1...i_d}^2$. It is now easy to see that $\E \widetilde{\tensor A}^{[k]} = \tensor A^{[k]}$. Thus, by applying Theorem \ref{thm::bound tensor (E norm(B) ^q)^(1/q)} with the parameter $\lambda_k$,
\begin{equation}
\label{ineq::expected norm of A^k - tilde A^k}
\begin{split}
\left(\E \norm{\tensor A^{[k]}-\widetilde{\tensor A}^{[k]}}_2^{q}\right)^{\frac{1}{q}} \leq
c 8^d \sqrt{2d \ln \left(\frac{5e}{\lambda_k} \right)} \left( \left[ \log_2 \frac{1}{\lambda_k} \right]^{d-1} \left( \sum_{j=1}^d \E \alpha_j^q \right)^{\frac{1}{q}} + \sqrt{\lambda_k n} \left( \E \beta^q \right)^{\frac{1}{q}} \right),
\end{split}
\end{equation}
where
$$
\alpha_j^2 \triangleq \max_{i_1,\ldots,i_{j-1},i_{j+1},\ldots,i_d} \left(\sum_{i_j=1}^n \widetilde{\tensor A}_{i_1 \ldots i_{j-1} i_j i_{j+1} ... i_d}^2\right) \quad \text{and} \quad \beta = \max_{i_1,...,i_d} | \widetilde{\tensor A}_{i_1...i_d}|.
$$
We now follow the same strategy as in Section~\ref{sxn:tterm1} in order to estimate the expectation terms in the right-hand side of the above inequality (i.e., we focus on the first term ($j=1$) only). First, note that
\begin{eqnarray*}
\E \max_{i_2,\ldots,i_d} \left(\sum_{i_1=1}^n \widetilde{\tensor A}_{i_1 \ldots i_d}^2\right)^{q/2} \leq
\sqrt{\E \max_{i_2,\ldots,i_d} \left(\sum_{i_1=1}^n \widetilde{\tensor A}^2_{i_1...i_d} \right)^q}.
\end{eqnarray*}
Let $S_{i_2...i_d} = \sum_{i_1} \widetilde{\tensor A}^2_{i_1...i_d}$. Then, using eqn.~(\ref{eqn:define_tildeAij}), the definition of $p_{i_1...i_d}$, and $\delta_{i_1...i_d}^2 = \delta_{i_1...i_d}$, we get
$$
S_{i_2...i_d} = \sum_{i_1=1}^n \frac{\delta_{i_1...i_d}}{p^2_{i_1 ...i_d}} \tensor A^2_{i_1...i_d} = \sum_{i_1=1}^n \delta_{i_1...i_d} \frac{\norm{\tensor A}^4_F}{\tensor A^4_{i_1...i_d} s^2} \tensor A^2_{i_1...i_d} = \sum_{i_1=1}^n \delta_{i_1...i_d} \frac{\norm{\tensor A}^4_F}{ s^2 \tensor A_{i_1...i_d}^2}.
$$
Using $\tensor A^2_{i_1...i_d} \geq \frac{2^{-k}\norm{\tensor A}^2_F}{s}$, we get $S_{i_2...i_d} \leq \frac{2^k \norm{\tensor A}^2_F}{ s} \left(\sum_{i_1} \delta_{i_1...i_d}\right)$, which leads to
\begin{equation}
\label{ineq::expected bound of max_i sum_j tilde(a)^2}
\E  \max_{i_2,\ldots,i_d} \sum_{i_1=1}^{n} \left(\widetilde{\tensor A}^2_{i_1 ...i_d} \right)^q = \E \max_{i_2,\ldots,i_d} S_{i_2...i_d}^q \leq \left( \frac{2^k \norm{\tensor A}^2_F}{ s} \right)^q  \E
\max_{i_2,\ldots,i_d} \left(\sum_{i_1=1}^{n} \delta_{i_1...i_d}\right)^q.
\end{equation}
We now seek a bound for the expectation $\E \max_{i_2,\ldots,i_d} \sum_{i_1} \left(\delta_{i_1...i_d} \right)^q$. The following lemma, whose proof may be found in the Appendix, provides such a bound.
\begin{lem}\label{lem::bound expected max_i sum_j delta_ij}
For any $q \geq 1$, we have
$$
\E \max_{i_2,\ldots,i_d} \left(\sum_{i_1=1}^n \delta_{i_1...i_d}\right)^q \leq 2 \left(5 n 2^{-k} + 2d\ln n + 2q \right)^q.
$$
\end{lem}
\noindent Combining Lemma~\ref{lem::bound expected max_i sum_j delta_ij} and equation (\ref{ineq::expected bound of max_i sum_j tilde(a)^2}), we obtain
$$
\E  \left(\max_{i_2,\ldots,i_d} \sum_{i_1=1}^{n} \widetilde{\tensor A}^2_{i_1...i_d}\right)^q \leq 2 \left( \frac{ (5n + 2 (d \ln n + q) 2^k) \norm{\tensor A}_F^2}{s} \right)^q.
$$
The same bound can be derived for all other terms in the first summand of eqn.~(\ref{ineq::expected norm of A^k - tilde A^k}). Thus,
\begin{equation*}
\begin{split}
\left( \sum_{j=1}^d \E \alpha_j^q \right)^{\frac{1}{q}} &\leq \left( \sum_{j=1}^d \E \left( \max_{i_1,\ldots,i_{j-1},i_{j+1},\ldots,i_d} \sum_{i_j=1}^n \widetilde{\tensor A}_{i_1 \ldots i_{j-1} i_j i_{j+1} ... i_d}^2\right)^q \right)^{\frac{1}{2q}} \\
&\leq (2d)^{\frac{1}{2q}} \sqrt{ \frac{(5n + (d \ln n+q) 2^{k+1}) \norm{\tensor A}_F^2}{s} }.
\end{split}
\end{equation*}

\noindent In addition, we have $\widetilde{\tensor A}^2_{i_1 ... i_d} \leq \frac{\tensor A^2_{i_1...i_d}}{p^2_{i_1...i_d}} \leq \frac{\norm{\tensor A}_F^4}{s^2 \tensor A^2_{i_1...i_d}} \leq \frac{2^k \norm{\tensor A}_F^2}{s} $. Thus,
$$
(\E \beta^q)^{\frac{1}{q}} \leq \sqrt{\frac{2^k \norm{\tensor A}^2_F}{s}}.
$$

\noindent Substituting these two inequalities into eqn.~(\ref{ineq::expected norm of A^k - tilde A^k}) we get the claim of the lemma.
\end{proof}

\subsection{Term~(\ref{eqn:step3}): bounding the tail}

We now focus on values of $k$ that exceed $\ell = \left\lfloor \log_2 \left(n^{d/2} /\ln^d n\right) \right\rfloor$ and prove the following lemma, which immediately provides a bound for term~(\ref{eqn:step3}).
\begin{lem}\label{lem:term3}
Using our notation,
$$\norm{\sum_{k = \ell+1}^{\infty} \left(\tensor A^{[k]} - \widetilde{\tensor A}^{[k]}\right)}_2 \leq \sqrt{\frac{n^{d/2} \ln^d n}{s}} \norm{\tensor A}_F.$$
\end{lem}
\begin{proof}
Intuitively, by the definition of $\tensor A^{[k]}$, we can observe that when $k$ is larger than $\ell = \left\lfloor \log_2 \left(n^{d/2} /\ln^d n\right) \right\rfloor$, the entries of $\tensor A^{[k]}$ are very small, whereas the entries of $\widetilde{\tensor A}^{[k]}$ are all set to zero during the second step of our sparsification algorithm. Formally, consider the sum
$$\tensor D = \sum_{k = \ell +1}^{\infty} \left(\tensor A^{[k]} - \widetilde{\tensor A}^{[k]}\right).$$
For all $k \geq \ell + 1 \geq \log_2 \left(n^{d/2} /\ln^d n\right)$, notice that the squares of all the entries of $\tensor A^{[k]}$ are at most $\frac{\ln^d n}{n^{d/2}} \frac{\norm{\tensor A}_F^2}{s}$ (by definition) and thus the tensors $\widetilde{\tensor A}^{[k]}$ are all-zero tensors. The above sum now reduces to
$$\tensor D = \sum_{k = \ell + 1}^{\infty} \tensor A^{[k]},$$
where the squares of all the entries of $\tensor D$ are at most $\frac{\ln^d n}{n^{d/2}} \frac{\norm{\tensor A}_F^2}{s}$. Since $\tensor D \in \R^{n \times...\times n}$, using $\norm{\tensor D}_2 \leq \norm{\tensor D}_F$, we immediately get
\begin{equation*}
\norm{\tensor D}_2 = \norm{\sum_{k = \ell + 1}^{\infty} \left(\tensor A^{[k]} - \widetilde{\tensor A}^{[k]}\right)}_2
\leq \sqrt{\sum_{i_1,i_2,\ldots,i_d=1}^n \tensor D_{i_1...i_d}^2} \leq \sqrt{\frac{n^{d/2} \ln^d n}{s}} \norm{\tensor A}_F.
\end{equation*}
\end{proof}

\subsection{Completing the proof of Theorem \ref{thm::main theorem of tensor sparsification}}

Theorem \ref{thm::main theorem of tensor sparsification} emerges by substituting Lemmas~\ref{lem::final bound of E norm B^1},~\ref{lem::final bound of E norm B^k}, and~\ref{lem:term3} to bound terms~(\ref{eqn:step1}),~(\ref{eqn:step2}), and~(\ref{eqn:step3}). We have
%
%\begin{eqnarray*}
%
\begin{equation}
\label{inq::proof Theorem 1 - 1st step}
\begin{split}
\left(\E\norm{\tensor A - \widetilde{\tensor A}}_2^{2d \ln n}\right)^{\frac{1}{2d\ln n}}
&\leq c_1 48^d d^{1/q+1/2} \sqrt{n}  \frac{\norm{\tensor A}_F}{\sqrt{s}} \\
&+ \sum_{k=2}^{\left\lfloor \log_2 \left(n^{d/2} /\ln^d n\right) \right\rfloor} c_2 8^d \sqrt{2d \ln \frac{5e}{\lambda_k} } \left( \log_2 \frac{1}{\lambda_k} \right)^{d-1} (2d)^{\frac{1}{2q}} \sqrt{5n + (d\ln n+q)2^{k+1} } \frac{\norm{\tensor A}_F}{\sqrt{s}} \\
&+ \sum_{k=2}^{\left\lfloor \log_2 \left(n^{d/2} /\ln^d n\right) \right\rfloor} c_2 8^d \sqrt{2d \ln \frac{5e}{\lambda_k}} \sqrt{2^k \lambda_k n}  \frac{\norm{\tensor A}_F}{\sqrt{s}} \\
&+ \sqrt{n^{d/2} \ln^d n} \frac{\norm{\tensor A}_F}{\sqrt{s}} \triangleq (M_1 + M_2 + M_3 + M_4) \frac{\norm{\tensor A}_F}{\sqrt{s}}.
\end{split}
\end{equation}
%
%\end{eqnarray*}
%
While the first term $M_1$ and the last term $M_4$ on the right-hand side are fixed, the second and third terms largely depends on the choice of parameters $\lambda_k$. We would want to select $\lambda_k$'s such that the right-hand side is as small as possible. For this task, we set
%$$
%\lambda_k \triangleq \frac{1}{16} \frac{1}{2^{2k/d}} \quad\quad \text{for } k = 2,3,...,\log_2 \frac{n^{d/2}}{\ln^d n}.
%$$
$$
\lambda_k \triangleq \frac{1}{n} \quad\quad \text{for } k = 2,3,...,\log_2 \frac{n^{d/2}}{\ln^d n}.
$$
Clearly, $\lambda_k \leq \frac{1}{64}$ as required by Theorem \ref{thm::bound tensor (E norm(B) ^q)^(1/q)}. In addition, the requirement $q \leq 2d \lambda_k n \ln \frac{5e}{\lambda_k}$ is always satisfied as long as $q \leq 2d \ln n$. We set $q \triangleq 2d \ln n$. This immediately implies that the quantity $d^{1/q}$ is bounded by a constant. Let $N \triangleq \left\lfloor \log_2 \frac{n^{d/2}}{\ln^d n} \right\rfloor$ to get
\begin{equation}
\label{inq::series bound}
\begin{split}
\sum_{k=2}^{N} \sqrt{2^k} = \sum_{k=1}^{\lfloor N/2 \rfloor} 2^k + 2^{1/2} \sum_{k=1}^{\lfloor (N-1)/2 \rfloor} 2^k &\leq (2^{\lfloor N/2 \rfloor + 1} -1) + 2^{1/2}  (2^{\lfloor (N-1)/2 \rfloor + 1} -1) \\	
&\leq 4 \times 2^{N/2} \leq 4 \sqrt{\frac{n^{d/2}}{\ln^d n}}
\end{split}
\end{equation}

\noindent Using the fact that $\sqrt{5n + (d\ln n+q) 2^{k+1} } \leq \sqrt{5n} + \sqrt{3d 2^{k+1} \ln n}$ and $\log_2 n \leq \ln n$, we can get the upper bound of $M_2$ as follows:

% assumption that $d\leq 0.5\ln n$; the fact that $n \leq n^{d/2}$ for all $d \geq 2$; and the fact that $2d \ln n \leq 2n$, since $n \geq 300$ and $d \leq 0.5 \ln n$. Then, by manipulating/removing constants we get:
\begin{equation}
\begin{split}
\nonumber
M_2 &\leq \sum_{k=2}^{\log_2 \left(n^{d/2} /\ln^d n \right) } c_4 8^d (2d \ln (5en))^{1/2} (\ln n)^{d-1} ( \sqrt{5n} + \sqrt{3d 2^{k+1} \ln n} ) \\
&\leq c_4 8^d \sqrt{d} \log_2 \left(\frac{n^{d/2}}{\ln^d n} \right) (\ln n)^{d-1/2} \sqrt{n} +  c_5 8^d d \ln^d n \sum_{k=2}^{\log_2 \left(n^{d/2} /\ln^d n \right) } 2^{k/2} \\
&\leq c_5 8^d d^{3/2} n^{1/2} (\ln n)^{d+1/2} + c_6 8^d d (\ln n)^d \sqrt{\frac{n^{d/2}}{\ln^d n} },
\end{split}
\end{equation}
where the last inequality is due to eqn.~(\ref{inq::series bound}). For $d \geq 3$, we derive an upper bound for $M_2$ as follows:
$$
M_2 \leq c_8 8^d d^{3/2} \sqrt{n^{d/2} \ln^d n} \sqrt{\max \left\{ 1, \frac{\ln^{d+1} n}{n^{d/2-1}} \right\} }.
$$

\noindent A similar bound can be derived for $M_3$:
\begin{equation}
\begin{split}
\nonumber
M_3 &= \sum_{k=2}^{\left\lfloor \log_2 \left(n^{d/2} /\ln^d n\right) \right\rfloor} c_2 8^d (2d \ln (5en))^{1/2} \sqrt{2^k} \leq c_9 8^d \sqrt{d} \sqrt{\frac{n^{d/2}}{\ln^{d-1} n}} .
\end{split}
\end{equation}

\noindent Combining the above results and substituting into (\ref{inq::proof Theorem 1 - 1st step}), we get
\begin{equation}
\begin{split}
\nonumber
\left(\E\norm{\tensor A - \widetilde{\tensor A}}_2^{2d\ln n}\right)^{\frac{1}{2d\ln n}} \leq c_{10} 20^d d^{3/2} \sqrt{\max \left\{ 1, \frac{\ln^{d+1} n}{n^{d/2-1}} \right\} } \sqrt{\frac{n^{d/2} \ln^d n}{s}}  \norm{\tensor A}_F,
\end{split}
\end{equation}
where the bound is due to the fact that $48^d \leq 20^d \sqrt{\ln^d n}$ for any $n \geq 320$. Applying Markov's inequality, we conclude that
$$
\norm{\tensor A - \widetilde{\tensor A}}_2 \leq c'_{10} 20^d \sqrt{\max \left\{ 1, \frac{\ln^{d+1} n}{n^{d/2-1}} \right\} } \sqrt{\frac{d^3 n^{d/2} \ln^d n}{s}}  \norm{\tensor A}_F
$$
holds with probability at least $1 - n^{-2d}$. The first part of Theorem~\ref{thm::main theorem of tensor sparsification} now follows by setting $s$ to the appropriate value.
For $d =2$, the upper bound for $M_2$ can be simplified:
$$
M_2 \leq c_7 8^d d \sqrt{n \ln^5 n}.
$$
Following the same steps as above, we also derive that
$$
\norm{\tensor A - \widetilde{\tensor A}}_2 \leq c'_{10} 20^d d \sqrt{\frac{n \ln^5 n}{s}}  \norm{\tensor A}_F
$$
holds with probability at least $1 - n^{-4}$. Theorem~\ref{thm::main theorem of tensor sparsification} now follows by setting $s$ to the appropriate value.
\section{Conclusions and open problems}

We presented the first provable bound for tensor sparsification with respect to the spectral norm. The main technical difficulty that we had to address in our work was the lack of measure concentration inequalities (analogous to the matrix-Bernstein and matrix-Chernoff bounds) for random tensors. To overcome this obstacle, we developed such an inequality using the so-called \textit{entropy-concentration tradeoff}. To the best of our knowledge, this is the first bound of its kind in the literature.

An interesting open problem would be to investigate whether there exist algorithms that, either deterministically or probabilistically, select elements of $\tensor A$ to include in $\tilde{\tensor A}$ and achieve much better accuracy than existing schemes. For example, notice that our algorithm, as well as prior ones, sample entries of $\tensor A$ with respect to their magnitudes; better sampling schemes might be possible. Improved accuracy will probably come at the expense of increased running time. Such algorithms would be very interesting from a mathematical and algorithmic viewpoint, since they will allow a better quantification of properties of a matrix/tensor in terms of its entries.

\vspace{0.02in}\noindent\textbf{Acknowledgements.} We would like to thank Prof. Dimitris Achlioptas for bringing~\cite{GT09} to our attention, as well as Tasos Zouzias for numerous useful discussions regarding our results. We also would be grateful to Prof. Roman Vershynin for the clarification regarding his papers. 

\bibliographystyle{plain} \bibliography{bibliog}
\section*{Appendix}

\noindent \textbf{Proof of Lemma \ref{lem::convert from probabilistic bound to expectation bound}.}

\begin{proof}
\textbf{(a)} From our assumption,
$$
\Prob(X \geq a + b (t + h)) \leq e^{- t}.
$$
Let $s = a + b (t + h)$. For any $q \geq 1$,
\begin{align*}
\E X^q &= \int_{0}^{\infty} \Prob(X \geq s) ds^q =  q \int_{0}^{\infty} \Prob(X \geq s) s^{q-1}  ds \\
&\leq q \int_{0}^{a + b h} s^{q-1} ds  + q \int_{a + b h}^{\infty} s^{q-1}  e^{- \frac{(s - a - b h)}{b}} ds.
\end{align*}
The first term in the above sum is equal to $\left(a + b h\right)^q$. The second term is somewhat harder to compute. We start by letting $g = a + b h$ and changing variables, thus getting
\begin{align*}
\int_{a + b h}^{\infty} s^{q-1}  e^{- \frac{(s - a - b h)}{b}} ds = b \int_0^{\infty} (g + b t)^{q-1} e^{-t} dt
&= b \sum_{i=0}^{q-1} \binom{q-1}{i} b^{q-1-i} g^i \int_0^{\infty} t^{q-1-i} e^{-t} dt.
\end{align*}
We can now integrate by parts and get
$$
\int_0^{\infty} t^{q-1-i} e^{-t} dt = (q - 1 - i) ! \leq q^{q - 1 - i} \quad \text{for all  } i = 0,...,q-1.
$$
Combining the above,
$$
q \int_{a + b h}^{\infty} s^{q-1}  e^{- \frac{(s - a - b h)}{b}} ds \leq qb \sum_{i=0}^{q-1} \binom{q-1}{i} (bq)^{q-1-i} g^i = qb (bq + g)^{q-1}.
$$
Finally,
$$
\E X^q \leq (a + b h)^q + bq (bq + g)^{q-1} \leq 2 (a + bh + bq)^q,
$$
which concludes the proof of the first part.

\noindent \textbf{(b)} From our assumption and since $t$ and $h$ are non-negative, we get
$$
\Prob(X \geq a + b (t + \sqrt{h})) \leq e^{- (t + \sqrt{h})^2 + h} \leq e^{-t^2}.
$$
Let $s = a + b \sqrt{h} + t b$. For any $q \geq 1$,
\begin{align*}
\E X^q &= \int_{0}^{\infty} \Prob(X \geq s) ds^q =  q \int_{0}^{\infty} \Prob(X \geq s) s^{q-1}  ds \\
&\leq q \int_{0}^{a + b \sqrt{h}} s^{q-1} ds  + q \int_{a + b \sqrt{h}}^{\infty} s^{q-1}  e^{- \frac{(s - a - b \sqrt{h})^2}{b^2}} ds.
\end{align*}
The first term in the above sum is equal to $\left(a + b \sqrt{h}\right)^q$. We now evaluate the second integral. Let $g = a + b \sqrt{h}$ and perform a change of variables to get
\begin{eqnarray*}
\int_{a + b \sqrt{h}}^{\infty} s^{q-1}  e^{- \frac{(s - a - b \sqrt{h})^2}{b^2}} ds &=& b \int_0^{\infty} (g + b t)^{q-1} e^{-t^2} dt \\
&=& b \sum_{i=0}^{q-1} \binom{q-1}{i} b^{q-1-i} g^i \int_0^{\infty} t^{q-1-i} e^{-t^2} dt.
\end{eqnarray*}
By integrating by parts we get (see below for a proof of eqn.~(\ref{eqn:integral_computation})):
\begin{equation}\label{eqn:integral_computation}
\int_0^{\infty} t^{q-1-i} e^{-t^2} dt \leq \sqrt{\frac{\pi}{2}} \left(\frac{q-1-i}{2}\right)^{(q-1-i)/2} \leq \sqrt{\frac{\pi}{2}} \left(\frac{q}{2}\right)^{(q-1-i)/2}.
\end{equation}
Thus, using $g = a + b\sqrt{h}$,
$$
\int_{a + b \sqrt{h}}^{\infty} s^{q-1}  e^{- \frac{(s - a - b \sqrt{h})^2}{b^2}} ds \leq b\sqrt{\frac{\pi}{2}} \sum_{i=0}^{q-1} \binom{q-1}{i} \left(b\sqrt{\frac{q}{2}}\right)^{q-1-i} g^i \leq \sqrt{2}b\left(a + b \sqrt{h} + b \sqrt{\frac{q}{2}}\right)^{q-1}.
$$
Finally, we conclude that
\begin{eqnarray*}
\E X^q &\leq& \left(a + b \sqrt{h}\right)^q + \sqrt{2}bq \left(a + b \sqrt{h} + b\sqrt{\frac{q}{2}}\right)^{q-1}\\
&\leq& \left(a + b \sqrt{h} + b\sqrt{\frac{q}{2}}\right)^q + \sqrt{4q} \left(a + b \sqrt{h} + b\sqrt{\frac{q}{2}}\right)^{q}\\
&\leq& 3\sqrt{q}\left(a + b \sqrt{h} + b\sqrt{\frac{q}{2}}\right)^q,
\end{eqnarray*}
which is the claim of the lemma. In the above we used the positivity of $a,b,$ and $h$ as well as the fact that $1 + \sqrt{4q} \leq 3\sqrt{q}$ for all $q \geq 1$.
\end{proof}

\noindent \textbf{Proof of eqn.~(\ref{eqn:integral_computation}).}

\begin{proof}
We now compute the integral $\int_{0}^{\infty} t^q e^{-t^2} dt$. Integrating by parts, we get
\begin{eqnarray*}
\int_{0}^{\infty} t^q e^{-t^2} dt &=& \frac{1}{2} \int_{0}^{\infty} - t^{q-1} d e^{-t^2}\\
& =& -\frac{1}{2} t^{q-1} e^{-t^2} |_0^{\infty} + \frac{1}{2} \int_{0}^{\infty} e^{-t^2} d t^{q-1} \\
&= &\frac{1}{2} (q-1) \int_{0}^{\infty} t^{q-2} e^{-t^2} dt.
\end{eqnarray*}
When $q$ is even, we get
$$
\int_{0}^{\infty} t^q e^{-t^2} dt = \left( \frac{1}{2} \right)^{q/2} (q-1) !! \int_{0}^{\infty} e^{-t^2} dt = \sqrt{\frac{\pi}{2}} \left( \frac{1}{2} \right)^{q/2} (q-1) !!.
$$
where $q !! = q (q-2) (q-4)\cdots $. If $q$ is odd, then
$$
\int_{0}^{\infty} t^q e^{-t^2} dt = \left( \frac{1}{2} \right)^{\lfloor q/2 \rfloor} (q-1) !! \int_{0}^{\infty} t e^{-t^2} dt = \left( \frac{1}{2} \right)^{\lfloor q/2 \rfloor + 1} (q-1) !! .
$$
We thus conlude
\begin{align*}
\int_{0}^{\infty} t^q e^{-t^2} dt \leq  \sqrt{\frac{\pi}{2}} \left( \frac{1}{2} \right)^{\lfloor q/2 \rfloor} (q-1) !! \leq \sqrt{\frac{\pi}{2}} \left( \frac{q-1}{2} \right)^{\lfloor q/2 \rfloor} \leq  \sqrt{\frac{\pi}{2}} \left( \frac{q-1}{2} \right)^{q/2} .
\end{align*}
\end{proof}

\noindent \textbf{Proof of Lemma \ref{prop::dicretize of norm of a tensor A}.}

\begin{proof}
We start by noting that every vector $z \in B$ can be written as $z = x + h$, where $x$ lies in $\N$ and $h \in \epsilon B$. Using the triangle inequality for the tensor spectral norm, we get
$$
\sup_{z \in B} \norm{\tensor A \times_1 z}_2 \leq \sup_{ x \in \N} \norm{\tensor A \times_1 x}_2 + \sup_{h \in \epsilon B} \norm{\tensor A \times_1 h}_2.
$$
\noindent It is now easy to bound the second term in the right-hand side of the above equation by $\epsilon \sup_{z \in B} \norm{\tensor A \times_1 z}_2$. Thus,
$$
\sup_{z \in B} \norm{\tensor A \times_1 z}_2 \leq \frac{1}{1-\epsilon} \sup_{x \in \N} \norm{\tensor A \times_1 x}_2.
$$
\noindent Repeating the same argument recursively for the tensor $\tensor A \times_1 x$ etc. we obtain the lemma.
\end{proof}

%\noindent \textbf{Proof of Lemma \ref{lem::bound sum of power of q}.}
%
%\begin{proof}
%%
%Let $a$ be the $n$-dimensional vector of the $a_i$'s. The lemma can be restated as follows:
%%
%\begin{eqnarray}\label{eqn:pdd2}
%%
%\left\|a\right\|_1^q &\leq& n^{q-1}\left\|a\right\|_{q}^q
%%
%\end{eqnarray}
%%
%The lemma obviously holds for $q=1$, so we only need to prove it for $q>1$. We will use the following property of vector norms: if $p > r > 0$, then $\left\|a\right\|_r \leq n^{\left(r^{-1}-p^{-1}\right)}\left\|a\right\|_p$. Applying it for $p=q>1$ and $r=1$, we get:
%%
%\begin{eqnarray}\label{eqn:pdd3}
%%
%\left\|a\right\|_1 &\leq& n^{1-q^{-1}}\left\|a\right\|_{q},
%%
%\end{eqnarray}
%%
%and the lemma follows immediately by raising both sides to the power $q$.
%%
%\end{proof}

\noindent \textbf{Proof of Lemma \ref{lem::bound expected max_i sum_j delta_ij}.}

\begin{proof}
Let $S = \max_{i_2,\ldots,i_d} \sum_{i_1=1}^n \delta_{i_1...i_d}$. We will first estimate the probability $\Prob (S \geq t)$ and then apply Lemma \ref{lem::convert from probabilistic bound to expectation bound} in order to bound the expectation $\E S^q$. Recall from the definition of $\delta_{i_1...i_d}$ that $\E \left(\delta_{i_1...i_d} - p_{i_1...i_d}\right)=0$ and let
$$X = \sum_{i_1=1}^n \left(\delta_{i_1...i_d} - p_{i_1...i_d}\right).$$
We will apply Bennett's inequality in order to bound $X$. Clearly $\abs{\delta_{i_1...i_d} - p_{i_1...i_d}} \leq 1$ and
\begin{align*}
\V (X) = \sum_{i_1=1}^n \V \left(\delta_{i_1...i_d} - p_{i_1...i_d}\right) &= \sum_{i_1=1}^n \E \left(\delta_{i_1...i_d} - p_{i_1...i_d}\right)^2 \\
&= \sum_{i_1=1}^n \left(p_{i_1...i_d}- p_{i_1...i_d}^2\right) \leq \sum_{i_1=1}^n p_{i_1...i_d}.
\end{align*}
Recalling the definition of $p_{i_1...i_d}$ and the bounds on the $\tensor A_{i_1...i_d}$'s, we get
$$
\V \left(X\right) \leq \sum_{i_1=1}^n \frac{s \tensor A^2_{i_1...i_d}}{\norm{\tensor A}^2_F} \leq n 2^{-(k-1)}.
$$
We can now apply Bennett's inequality in order to get
$$
\Prob( X > t ) = \Prob\left( \sum_{i_1=1}^n \delta_{i_1...i_d} > \sum_{i_1=1}^n p_{i_1...i_d} + t\right ) \leq e^{-t/2},
$$
for any $t \geq 3 n 2^{-(k-1)} / 2$. Thus, with probability at least $1-e^{-t/2}$,
$$
\sum_{i_1=1}^n \delta_{i_1...i_d} \leq  n 2^{-(k-1)} + t,
$$
since $\sum_{i_1=1}^n p_{i_1...i_d} \leq n 2^{-(k-1)}$. Setting $t = \left(3 n 2^{-(k-1)}/2\right) + 2\tau$ for any $\tau \geq 0$ we get
$$\Prob \left( \sum_{i_1=1}^n \delta_{i_1...i_d} \geq  \frac{5}{2} n 2^{-(k-1)} + 2 \tau \right) \leq e^{-\tau}.$$
Taking a union bound yields
$$
\Prob \left( \max_{i_2,\ldots,i_d} \sum_{i_1=1}^n \delta_{i_1...i_d} \geq  5 n 2^{-k} + 2 \tau \right) \leq n^{d-1}e^{-\tau} = e^{-\tau + (d-1)\ln n},
$$
where the $n^{d-1}$ term appears because of all possible choices for the indices $i_2,\ldots,i_d$. Applying Lemma \ref{lem::convert from probabilistic bound to expectation bound} with $a = 5 n 2^{-k}$, $b = 2$, and $h = (d-1)\ln n$, we get
\begin{equation}\label{eqn:ppp111}
\E \left(\max_{i_2,\ldots,i_d} \sum_{i_1=1}^n \delta_{i_1...i_d}\right)^q \leq 2 \left(5 n 2^{-k} + 2(d-1)\ln n + 2 q\right)^q \leq
2 \left(5 n 2^{-k} + 2d\ln n + 2q \right)^q.
\end{equation}
The proof is completed.

%We now note that, clearly, $5 n 2^{-k} \leq 5 n^{d/2} 2^{-k}$ for all $d \geq 2$. Also, using our assumption $d \leq 0.5\ln n$ as well as $k \leq \left\lfloor \log_2 \left(n^{d/2}/\ln^2 n\right)\right \rfloor$, we can prove that
%%
%$$2d\ln n \leq n^{d/2}2^{-k}.$$
%%
%Substituting the above bounds in eqn.~(\ref{eqn:ppp111}) concludes the proof of the lemma.

\end{proof} 

\end{document}